\documentclass[a4paper,11pt,reqno,english]{smfart}

\usepackage{enumerate}
\usepackage{amssymb,amsmath,latexsym,amsthm}
\usepackage[T1]{fontenc}
\usepackage[a4paper, tmargin=1.5in, bmargin=1.5in, lmargin=1in,rmargin=1in]{geometry}
\usepackage{url}
\usepackage[frenchb, main=english]{babel}
\usepackage[utf8]{inputenc}
\usepackage{mathrsfs}
\usepackage{xcolor}
\usepackage{bigints}
\usepackage{comment}
\definecolor{violet}{rgb}{0.0,0.2,0.7}
\definecolor{rouge2}{rgb}{0.8,0.0,0.2}
\usepackage{tikz}
\usepackage{empheq}
\usepackage{tikz-cd}
\usetikzlibrary{matrix,arrows,decorations.pathmorphing}
\usepackage{hyperref}
\usepackage{mathpazo} 
\hypersetup{
	bookmarks=true,         
	unicode=false,          
	pdftoolbar=true,        
	pdfmenubar=true,        
	pdffitwindow=false,     
	pdfstartview={FitH},    
	pdftitle={},    
	pdfauthor={},     
	colorlinks=true,       
	linkcolor=rouge2,          
	citecolor=violet,        
	filecolor=black,      
	urlcolor=cyan}           
\usepackage{enumitem}
\usepackage{appendix}

\usepackage{mathtools}
\usepackage{dsfont}

\theoremstyle{plain}
    \newtheorem{thm}{Theorem}[section]
	
	\newtheorem{lem}[thm]{Lemma}
	\newtheorem{prop}[thm]{Proposition}
	
	\newtheorem{cor}[thm]{Corollary}
	
\theoremstyle{plain}
	\newtheorem{bigthm}{Theorem}
	\newtheorem{bigprop}[bigthm]{Proposition}

	\newtheorem{taggedbigsetx}{Setting} 
	\newenvironment{taggedbigset}[1]
    {\renewcommand\thetaggedbigsetx{#1}\taggedbigsetx}
    {\endtaggedbigsetx}

    \newtheorem*{bigrmk*}{Remark}

\theoremstyle{definition}
	\newtheorem{defn}[thm]{Definition}
	
	\newtheorem{nota}[thm]{Notation}

    \newtheorem{claim}[thm]{Claim}
    \newtheorem*{claim*}{Claim}
    
    \newtheorem*{ques*}{Question}
	\newtheorem*{ack*}{Acknowledgements}
\theoremstyle{remark}
	\newtheorem{rmk}[thm]{Remark}
	\newtheorem*{rmk*}{Remark}

	\newtheorem*{ans*}{Answer}

\numberwithin{equation}{section}
\renewcommand{\labelenumi}{(\roman{enumi})}
\newlist{steps}{enumerate}{1}
\setlist[steps, 1]{label = Step \arabic*:}

\mathcode`l="8000
\begingroup
\makeatletter
\lccode`\~=`\l
\DeclareMathSymbol{\lsb@l}{\mathalpha}{letters}{`l}
\lowercase{\gdef~{\ifnum\the\mathgroup=\m@ne \ell \else \lsb@l \fi}}%
\endgroup

\DeclareFontFamily{U}{MnSymbolC}{}
\DeclareSymbolFont{MnSyC}{U}{MnSymbolC}{m}{n}
\DeclareFontShape{U}{MnSymbolC}{m}{n}{
	<-6>  MnSymbolC5
	<6-7>  MnSymbolC6
	<7-8>  MnSymbolC7
	<8-9>  MnSymbolC8
	<9-10> MnSymbolC9
	<10-12> MnSymbolC10
	<12->   MnSymbolC12}{}
\DeclareMathSymbol{\intprod}{\mathbin}{MnSyC}{'270}
\DeclareMathOperator{\Id}{Id}

\DeclareMathOperator{\im}{im}
\DeclareMathOperator{\tr}{tr}
\DeclareMathOperator{\pr}{pr}
\DeclareMathOperator{\Vol}{Vol}
\DeclareMathOperator{\Ric}{Ric}

\DeclareMathOperator{\supp}{supp}

\DeclareMathOperator{\dist}{dist}

\DeclareMathOperator{\PH}{PH}

\DeclareMathOperator{\PSH}{PSH}
\DeclareMathOperator{\MA}{MA}
\DeclareMathOperator{\CAP}{Cap}

\DeclareMathOperator*{\osc}{osc}

\DeclareMathOperator{\Exc}{Exc}

\DeclareMathOperator{\OT}{OT}

\DeclareMathOperator{\KE}{KE}
\DeclareMathOperator{\Aut}{Aut}
\DeclareMathOperator{\mult}{mult}

\def\1{\mathds{1}}

\def\D{\mathbf{D}}
\def\E{\mathbf{E}}

\def\L{\mathbf{L}}

\newcommand{\ii}{\mathrm{i}}

\newcommand{\loc}{\mathrm{loc}}

\newcommand{\nmlz}{{\mathrm{norm}}}
\newcommand{\fibre}{{\mathrm{fibre}}}

\newcommand{\wX}{{\widetilde{X}}}

\newcommand{\wom}{{\widetilde{\omega}}}
\newcommand{\wvph}{{\widetilde{\varphi}}}

\newcommand{\wu}{{\widetilde{u}}}
\newcommand{\wv}{{\widetilde{v}}}

\newcommand{\bBD}{\overline{\mathbb{D}}}

\newcommand{\hu}{{\widehat{u}}}

\newcommand\sm{\sigma}
\newcommand\dt{\delta}

\newcommand\vep{\varepsilon}
\newcommand\vph{\varphi}

\newcommand\om{\omega}
\newcommand\ta{\theta}
\newcommand\gm{\gamma}

\newcommand\af{\alpha}
\newcommand\bt{\beta}
\newcommand\ld{\lambda}
\newcommand\zt{\zeta}

\newcommand\Om{\Omega}

\newcommand\Gm{\Gamma}

\newcommand\Ta{\Theta}

\newcommand\BN{\mathbb{N}}

\newcommand\BQ{\mathbb{Q}}
\newcommand\BR{\mathbb{R}}
\newcommand\BC{\mathbb{C}}
\newcommand\BB{\mathbb{B}}

\newcommand\BD{\mathbb{D}}
\newcommand\BP{\mathbb{P}}

\newcommand\CB{\mathcal{B}}
\newcommand\CC{\mathcal{C}}
\newcommand\CD{\mathcal{D}}
\newcommand\CE{\mathcal{E}}

\newcommand\CI{\mathcal{I}}

\newcommand\CL{\mathcal{L}}

\newcommand\CO{\mathcal{O}}

\newcommand\CU{\mathcal{U}}
\newcommand\CV{\mathcal{V}}
\newcommand\CW{\mathcal{W}}
\newcommand\CX{\mathcal{X}}
\newcommand\CY{\mathcal{Y}}
\newcommand\CZ{\mathcal{Z}}

\newcommand\lt{\left}
\newcommand\rt{\right}

\newcommand\pl{\partial}
\newcommand\db{\bar{\partial}}

\newcommand\ddb{\partial \bar{\partial}}

\newcommand\dd{\mathrm{d}}
\newcommand\dc{\mathrm{d}^{\mathrm{c}}}
\newcommand\ddc{\mathrm{d}\mathrm{d}^{\mathrm{c}}}

\newcommand\norm[1]{\left\lVert {#1} \right\rVert}

\newcommand\abs[1]{\left\lvert {#1} \right\rvert}

\newcommand\w{\wedge}

\newcommand\reg{\mathrm{reg}}
\newcommand\sing{\mathrm{sing}}

\newcommand\set[2]{\left\{ {#1} \, \middle| \, {#2} \right\}}
\newcommand\iprod[2]{\left\langle {#1}, {#2} \right\rangle}

\newcommand{\npp}[1]{{\left\langle #1 \right\rangle}}

\newcommand\res[2]{\left. {#1} \right|_{#2}} 

\newcommand{\RN}[1]{\textup{\uppercase\expandafter{\romannumeral#1}}}

\makeatletter
\newsavebox{\@brx}
\newcommand{\llangle}[1][]{\savebox{\@brx}{\(\m@th{#1\langle}\)}%
  \mathopen{\copy\@brx\kern-0.5\wd\@brx\usebox{\@brx}}}
\newcommand{\rrangle}[1][]{\savebox{\@brx}{\(\m@th{#1\rangle}\)}%
  \mathclose{\copy\@brx\kern-0.5\wd\@brx\usebox{\@brx}}}
\makeatother

\title{K\"ahler--Einstein metrics on families of Fano varieties}

\author{Chung-Ming Pan}
\address[Chung-Ming Pan]{Simons Laufer Mathematical Sciences Institute; 17 Gauss Way, Berkeley, CA 94720-5070, USA \qquad\qquad\qquad\qquad\qquad\qquad\qquad\qquad\qquad\qquad\qquad}
\email{\href{mailto:bandan770@gmail.com}{bandan770@gmail.com} \qquad\qquad\qquad\qquad\qquad\qquad\qquad\qquad\qquad\qquad\qquad}
\urladdr{\href{https://chungmingpan.github.io/}{https://chungmingpan.github.io/}}

\author{Antonio Trusiani}
\address[Antonio Trusiani]{Chalmers University of Technology; Chalmers tv\"argata 3, 41258 G\"oteborg, Sweden \qquad\qquad\qquad\qquad\qquad\qquad\qquad\qquad\qquad\qquad\qquad}
\email{\href{mailto:trusiani@chalmers.se}{trusiani@chalmers.se} \qquad\qquad\qquad\qquad\qquad\qquad\qquad\qquad\qquad\qquad\qquad}
\urladdr{\href{https://sites.google.com/view/antonio-trusiani/}{https://sites.google.com/view/antonio-trusiani/}}

\date{\today}

\begin{document} 

\maketitle

\begin{abstract}
Given a one-parameter family of $\BQ$-Fano varieties such that the central fibre admits a unique K\"ahler--Einstein metric, we provide an analytic method to show that the neighboring fibre admits a unique K\"ahler--Einstein metric. 
Our results go beyond by establishing uniform a priori estimates on the K\"ahler--Einstein potentials along fully degenerate families of $\BQ$-Fano varieties. 
In addition, we show the continuous variation of these K\"ahler--Einstein currents, and establish uniform Moser--Trudinger inequalities and uniform coercivity of the Ding functionals. 
Central to our article is introducing and studying a notion of convergence for quasi-plurisubharmonic functions within families of normal K\"ahler varieties. 
We show that the Monge--Amp\`ere energy is upper semi-continuous with respect to this topology, and we establish a Demailly--Koll\'ar result for functions with full Monge--Amp\`ere mass.
\end{abstract}

\tableofcontents

\section*{Introduction}
Finding a canonical metric on a compact K\"ahler manifold is a central topic in complex and differential geometry. 
K\"ahler--Einstein metrics are prototypical examples of canonical metrics. 
On a compact K\"ahler manifold $X$, a K\"ahler--Einstein metric $\om_{\KE}$ is a K\"ahler metric whose Ricci form is proportional to the K\"ahler form; namely $\Ric(\om_{\KE}) = \ld \om_{\KE}$ for some $\ld \in \{-1,0,1\}$.
The notion of K\"ahler--Einstein metrics still makes sense on mildly singular varieties (cf. \cite{EGZ_2009}). 
Such a metric is a genuine K\"ahler--Einstein metric on the smooth locus and has "bounded potential" near the singular set. 
Constructing a K\"ahler--Einstein metric boils down to solving a complex Monge--Amp\`ere equation: 
\[
    (\om + \ddc \vph)^n = e^{-\ld \vph} f \om^n
    \quad\text{with }
    \vph \in \PSH(X,\om) \cap L^\infty(X),
\]
where $\om$ is a fixed smooth metric in the desired cohomology class, $f$ is an $L^p$ density function for some $p > 1$, and $\vph$ is the unknown. 
We denoted by $\PSH(X,\om)$ the set of all $\omega$\emph{-plurisubharmonic functions} (see Section~\ref{sec:prelim} for the precise definition).

\smallskip
The solution to the (singular) Calabi conjecture \cite{Yau_1978, EGZ_2009} provides a powerful existence result of K\"ahler--Einstein metrics when the Ricci curvature is zero or negative, i.e. $\ld = -1$ or $0$.
However, if $X$ is Fano ($\ld$ > 0), the existence of K\"ahler--Einstein metrics in the anti-canonical class is more subtle.
In 2015, Chen, Donaldson and Sun \cite{CDS_2015_1, CDS_2015_2, CDS_2015_3} proved that a Fano manifold $X$ admits K\"ahler--Einstein metrics if and only if $X$ is K-(poly)stable.  

\smallskip
In birational geometry, Fano varieties and their families are central objects. 
Recall that a Fano variety is a normal variety such that $-K_X$ is a $\BQ$-ample line bundle (these are also called $\BQ$-Fano varieties in the literature).
They often have rich geometry as they could have many interesting birational models, and they are also terminal objects in the Minimal Model Program. 
From the moduli space point of view, singular Fano varieties naturally appear as degenerations of Fano manifolds.
The notion of K-stability plays a key role in recent developments in constructing well-behaved moduli spaces of Fano varieties (cf. \cite{Odaka_2015, Spotti_Sun_Yao_2016, Li_Wang_Xu_2019, Blum_Xu_2019}). 
It is natural to wonder about the behavior of K\"ahler--Einstein metrics in a family of K-stable Fano varieties.

\smallskip
Spotti, Sun, and Yao \cite{Spotti_Sun_Yao_2016} have shown that if a family $\pi: \CX \to \BD$ is a $\BQ$-Gorenstein, projective, \emph{smoothing} of a K-stable $\BQ$-Fano variety $X_0$, then the K\"ahler--Einstein potentials are uniformly bounded over a small neighborhood near $0$. 
Moreover, the K\"ahler--Einstein metrics vary smoothly away from the singular set.
Their approach relies on careful estimates regarding cone metrics, which are one of the main arguments of the original proof of Chen--Donaldson--Sun.

\smallskip
On the other hand, a variational approach \cite{BBGZ_2013,BBEGZ_2019,BBJ_2021} based on pluripotential theory plays an important role in constructing K\"ahler--Einstein metrics, and it also helps to simplify the proof of Chen--Donaldson--Sun. One of the strengths of this approach relies on its adaptability. On singular Fano varieties, Li--Tian--Wang \cite{Li_Tian_Wang_2021, Li_Tian_Wang_2022} and Li \cite{Li_2022} generalized the Yau--Tian--Donaldson correspondence by such a variational argument. 

\smallskip
The purpose of this article is to develop a relative version of pluripotential theory to study K\"ahler--Einstein metrics on families of Fano varieties. 
Our approach involves introducing a concept of convergence for quasi-plurisubharmonic functions within families and extending certain theorems of pluripotential theory to this context. 
Through a variational argument, we provide an analytic proof of the openness of the existence of K\"ahler--Einstein metrics.
Most significantly, we establish uniform a priori estimates of K\"ahler--Einstein potentials for fully degenerate families of Fano varieties.

\subsection*{Openness}
Our first result is the following (see Theorem~\ref{thm:openness_KE}): 

\begin{bigthm}\label{bigthm:openness_KE}
Let $\CX$ be an $(n+1)$-dimensional $\BQ$-Gorenstein variety and let $\pi: \CX \to \BD$ be a proper holomorphic surjective map with connected fibres.
Assume that $-K_{\CX/\BD}$ is relatively ample, $X_0$ is klt, and $\Aut(X_0)$ is discrete.
If the central fibre $X_0$ admits a K\"ahler--Einstein metric, so do the nearby fibres.
\end{bigthm}

When $\pi: \CX \to \BD $ is smooth, Theorem \ref{bigthm:openness_KE} is known by Koiso \cite[Prop.~10.1]{Koiso_1983}. 
Spotti, Sun and Yao \cite[Thm.~1.1]{Spotti_Sun_Yao_2016} proved Theorem \ref{bigthm:openness_KE} when $\pi:\CX\to \BD$ is a smoothing of the central fibre.
Let us stress that we do not make any regularity assumptions on the fibres. 
Furthermore, the discreteness of the automorphism group on the central fibre is an important hypothesis as there are arbitrary small deformations of the Mukai--Umemura threefold which do not admit K\"ahler--Einstein metrics (see e.g. \cite[Cor.~1.3]{Tian_1997} and \cite[Cor.~1]{Donaldson_2007}). 
Also, since $\Aut(X_0)$ is discrete, the K\"ahler--Einstein metric is unique on each fibre close to $X_0$.

\smallskip
By the singular Yau--Tian--Donaldson correspondence \cite{Li_Tian_Wang_2022, Liu_Xu_Zhuang_2022}, a $\BQ$-Fano variety with discrete automorphism group 
admits a K\"ahler--Einstein metric if and only if it is K-stable. 
Consequently, the openness of a unique K\"ahler--Einstein metric can be derived by combining two deep works: the Yau--Tian Donaldson correspondence and the (Zariski) openness of K-stability proved by Blum and Liu \cite{Blum_Liu_2022}. 
Notably, our analytic approach to obtain Euclidean openness does not rely on the aforementioned two involved results. 

\subsection*{The uniform estimate}
The second and more important theme of this article is a uniform $L^\infty$-estimate of K\"ahler--Einstein potentials on degenerate families of $\BQ$-Fano varieties and a natural continuity result of K\"ahler--Einstein currents.

\smallskip
Establishing a uniform $L^\infty$-estimate of the K\"ahler--Einstein potentials solving the complex Monge--Amp\`ere equations is crucial in finding K\"ahler--Einstein metrics. 
Based on the pluripotential idea of Ko{\l}odziej \cite{Kolodziej_1998} (see also \cite{EGZ_2009}), Di Nezza, Guedj, and Guenancia \cite{DGG2020} obtained a very precise $L^\infty$-estimate of K\"ahler--Einstein potentials in cases of Calabi--Yau families and families of manifolds of general type.
In the Fano case, as explained before, the situation is more involved. 
We consider families of $\BQ$-Fano varieties as in Theorem~\ref{bigthm:openness_KE} and ask how the bound of the K\"ahler--Einstein potentials varies (see Theorem~\ref{thm:L_infty_est} and Corollary~\ref{cor:conti_KE_current}). 

\begin{bigthm}\label{bigthm:L_infty_est}
Under the setting of Theorem~\ref{bigthm:openness_KE},
let $\omega\in c_1(- K_{\CX/\BD})$ be a K\"ahler metric. 
Then there exists $C_{\MA}>0$ such that for all $t$ sufficiently close to $0$,
\[
    \osc_{X_t} \vph_{\KE,t} \leq C_{\MA},
\]
where $\om_{\KE,t}:=\om_{|X_t} + \ddc_t \varphi_{\KE,t}$ is the K\"ahler--Einstein metric on $X_t$.
    
In addition, $\vph_{\KE,t}\in \PSH(X_t,\om_t)$ varies smoothly away from the singular locus, and for all $k \in \{0,1, \cdots, n\}$, the following map is continuous near $0 \in \BD$
\[
    t \longmapsto \om_{\KE,t}^k \w [X_{t}] \in (\CD_{(n-k,n-k)}(\CX))',
\]
where $(\CD_{(n-k,n-k)}(\CX))'$ is the space of currents of bidimension $(n-k,n-k)$ (bidegree $(k+1,k+1)$).
\end{bigthm}

In Theorem~\ref{bigthm:L_infty_est}, we denoted by $X_t := \pi^{-1}(t)$ the (schematic) fibre over $t \in \BD$. 
By Theorem~\ref{bigthm:openness_KE} the $\BQ$-Fano variety $X_t$ admits a unique K\"ahler--Einstein metric for all $t$ close to $0$. 
We normalize the K\"ahler--Einstein potential $\varphi_{\KE,t}\in \PSH(X_t,\om_t)$ so that it solves the complex Monge--Amp\`ere equation
\[
    \frac{1}{V}(\om_t + \ddc_t \vph_{\KE,t})^n = e^{-\vph_{\KE,t}} \mu_t
\]
where $V = \int_{X_t} \om_t^n$ is the volume of each fibre $X_t$, and $\mu_t$ is the adapted probability measure (see Definition~\ref{defn:adapted_measure}). 

\smallskip
Comparing with the result of Spotti--Sun--Yao \cite[Thm.~1.3]{Spotti_Sun_Yao_2016}, in Theorem~\ref{bigthm:L_infty_est}, we do not require the family to be a smoothing of the central fibre. 
Furthermore, our pluripotential proof does not rely on the cone angles approach. We also extend Theorem \ref{bigthm:L_infty_est} to log Fano pairs (see Theorem~\ref{thm:L_infty_log}). 
Namely, our method covers the part of uniform estimate in {\it loc. cit.} in more general situations.
For the stream of the article, we do not state the log version here (see Section~\ref{sec:log pairs} for details). 

\smallskip
Related to the previous Theorems~\ref{bigthm:openness_KE} and \ref{bigthm:L_infty_est} with the same hypothesis, we also prove a uniform version of Moser--Trudinger inequalities (see Proposition \ref{prop:Moser--Trudinger}). 
To the authors' knowledge, this result represents the first solution in families to Aubin's "hypoth\`ese fondamentale" \cite{Aubin_1984}.

\subsection*{Highlights of the proofs}
We introduce a notion of the $L^1$ (resp. $\CC^0, \CC^\infty$) \emph{convergence in families} to establish a relative version of pluripotential theory and variational approach. 
See Definition~\ref{defn:conv_in_family} for more details. 
We underline that the $L^1$-convergence in families is essentially equivalent to the convergence of the currents $\big(\om_t+\ddc_t u_t)\wedge [X_t]$ to $\big(\om_0+\ddc_0 u_0)\wedge [X_0]$ (see Proposition~\ref{prop:hartogs_currents}).
We also establish two important results in classical pluripotential theory along sequences convergent in the family sense. 
They represent the key ingredients in the proofs of Theorems \ref{bigthm:openness_KE} and \ref{bigthm:L_infty_est}.

\smallskip

On a fixed normal compact K\"ahler variety $(X,\om)$, the Monge--Amp\`ere energy functional $\E: \PSH(X,\om) \to \BR \cup \{-\infty\}$ is a primitive of the complex Monge--Amp\`ere operator (see Section \ref{sec:prelim} for more information). We obtain the upper semi-continuity of Monge--Amp\`ere energies along a sequence $L^1$ converging in families (see Proposition~\ref{prop:usc_energy_family}). 

\begin{bigprop}\label{bigprop:usc_energy}
Let $\pi: \CX \to \BD$ be a proper surjective holomorphic map with connected fibres, which forms a family of varieties. Assume that $X_0$ is normal, and let $\om$ be a K\"ahler metric given by the curvature of a positive $\BQ$-line bundle $\CL$ over $\CX$. If $u_k \in \PSH(X_{t_k}, \om_{t_k})$ $L^1$-converges to $u_0 \in \PSH(X_0, \om_0)$ then
\[
    \limsup_{k \to +\infty} \E_{t_k}(u_k) \leq \E_0 (u_0).
\]
\end{bigprop}

We denoted by $\E_t$ the Monge--Amp\`ere energy functional on the fibre $X_t$ with respect to $\om_t$. 
Let us stress that the Monge--Amp\`ere operator is not continuous along an $L^1$-convergent sequence (cf. \cite{Lelong_1983}). 
On a fixed variety, the upper semi-continuity of Monge--Amp\`ere energy is a consequence of the monotonicity of $\E$ and an envelope argument (see e.g. \cite[Prop.~10.19]{GZbook}). 
However, in the family setting, several difficulties appear in comparing functions on different fibres, varying complex structures and singularities. 
The proof consists in constructing and studying well-behaved algebraic approximations by the Bergman kernels of the line bundle, and then controlling the difference between algebraic approximations and the original sequence.

\smallskip
The second element is a Demailly--Koll\'ar type result along sequences of functions in the class of full Monge--Amp\`ere potentials (class $\CE$, see Section~\ref{sec:prelim} for the definition) converging in families (see Proposition~\ref{prop:Demailly_Kollar_family}): 

\begin{bigprop}\label{bigprop:DK}
Let $\pi: \CX \to \BD$ be a proper surjective holomorphic map with connected fibres, which forms a family of varieties. Assume that $X_0$ is normal, and let $\om$ be a K\"ahler form on $\CX$. If $u_k \in \CE(X_{t_k}, \om_{t_k})$ $L^1$-converges to $u_0 \in \CE(X_0, \om_0)$ then for all $\gm > 0$,
\[
    \int_{X_{t_k}} e^{-\gm u_k} \om_{t_k}^n 
    \xrightarrow[k \to + \infty]{} \int_{X_0} e^{-\gm u_0} \om_0^n.
\]
In addition, if $\pi:\CX \to \BD$ is a $\BQ$-Gorenstein family and $X_0$ has only klt singularities, then for all $\gm > 0$,
\[
    \int_{X_{t_k}} e^{-\gm u_k} \dd\mu_{t_k}
    \xrightarrow[k \to +\infty]{} \int_{X_0} e^{-\gm u_0} \dd\mu_0,
\]
where $\mu_t$ is an adapted measure on $X_t$ (see Section~\ref{sec:adapted_measure}).
\end{bigprop}

The proof of the main result of Demailly--Koll\'ar \cite{Demailly_Kollar_2001} heavily relies on the Ohsawa--Takegoshi theorem, which is not available on singular varieties. 
Thus, Proposition~\ref{bigprop:DK} is not trivial even on a fixed fibre. 
We strongly use here the assumption of full Monge--Amp\`ere masses.

\subsection*{Organization of the article}
Section~\ref{sec:prelim} is a quick overview of pluripotential notions throughout the article. 
In Section~\ref{sec:conv_fami}, we introduce a concept of {\it convergence in families} and study Hartogs type properties along a convergent sequence in this sense. 
Section~\ref{sec:usc_energy} and Section~\ref{sec:DK_fami} are devoted to prove Proposition~\ref{bigprop:usc_energy} and Proposition~\ref{bigprop:DK}, respectively. 
Finally, in Section~\ref{sec:applications}, we combine the elements developed in the previous sections to establish Theorems~\ref{bigthm:openness_KE}~and~\ref{bigthm:L_infty_est}. 

\begin{ack*}
The authors are grateful to Vincent Guedj and Henri Guenancia for their constant support, suggestions, and encouragement.
The authors would like to thank Quang-Tuan Dang for carefully reading the first draft.
The authors are indebted to the anonymous referee for the valuable comments. 
C.-M. Pan is partially supported by the EUR MINT project ANR-18-EURE-0023 and the research project HERMETIC ANR-11-LABX-0040.
A. Trusiani is supported by the grant KAW 2020.0281 from the Knut and Alice Wallenberg Foundation.
\end{ack*}

\section{Preliminaries}\label{sec:prelim}
In this section, we review definitions and some properties of the variational approach that we are going to use in the sequel. 
We define the twisted exterior derivative $\dc := \frac{\ii}{4\pi} (\db - \pl)$ so that $\ddc = \frac{\ii}{2\pi} \ddb$. 
Let $(X,\om)$ be an $n$-dimensional compact normal K\"ahler variety. By variety, we mean an irreducible reduced complex analytic space. 

\subsection{Monge--Amp\`ere energy}\label{sec:prelim_MA_energy}
Denote by $\PSH(X,\om)$ the set of $\om$-plurisubharmonic ($\om$-psh) functions which are not identically $-\infty$. Elements $u\in\PSH(X,\om)$ are locally given as a sum of a plurisubharmonic function and a smooth function, and they satisfy $\omega+\ddc u\geq 0$ in the sense of currents (see \cite{Demailly_1985, EGZ_2009} for details on singular varieties).
By Bedford--Taylor's theory \cite{Bedford_Taylor_1982}, the Monge--Amp\`ere operator can be extended to globally bounded $\om$-psh functions on smooth domains. 
In the singular setting, the Monge--Amp\`ere operator of locally bounded psh functions can also be defined by taking zero through singular locus (cf. \cite{Demailly_1985}). 

\smallskip
For all $u \in \PSH(X,\om) \cap L^\infty(X)$, the Monge--Amp\`ere energy is explicitly given by 
\[
    \E(u) := \frac{1}{(n+1)V} \sum_{j=0}^n \int_X u (\om + \ddc u)^{j} \w \om^{n-j}
\]
where $V := \int_X \om^n$ is the volume of $X$ with respect to $\om$.
One can derive the following elementary properties:
\begin{itemize}
    \item for $u \in \PSH(X,\om) \cap L^\infty(X)$ and $c \in \BR$, $\E(u+c) = \E(u) + c$;
    \item for $u, v \in \PSH(X,\om) \cap L^\infty(X)$, if $u \leq v$, then $\E(u) \leq \E(v)$ and with equality if and only if $u=v$.
\end{itemize}
By the second property, $\E$ admits a unique extension to the whole $\PSH(X, \om)$ defined as
\[
    \E(u) := \inf\set{\E(v)}{u \leq v, \, v \in \PSH(X,\om) \cap L^\infty(X)} \in \BR \cup \{-\infty\}.
\]
Following the same proof as in \cite[Prop.~10.19]{GZbook}, one can check that such an extension $\E$ is still monotone and upper semi-continuous with respect to the $L^1$-convergence (see \cite[Page 36]{BBEGZ_2019}).
We set 
\[
    \CE^1(X,\om) := \set{u \in \PSH(X,\om)}{\E(u) > -\infty}
\]
for the finite (Monge--Amp\`ere) energy class.

\smallskip

The class $\CE(X,\om)$ of $\om$-psh functions with full Monge--Amp\`ere mass is defined as 
\[
    \CE(X,\om) := \set{u \in \PSH(X,\om)}{\lim_{j \to -\infty} {}^{\uparrow} \int_X \1_{\{u>-j\}} (\om + \ddc \max\{u,-j\})^n = V}.
\]
Note that one has $\CE^1(X,\om) \subset \CE(X,\om)$ by an argument similar to \cite[Prop. 10.16]{GZbook}.

\begin{rmk}
The reference \cite{GZbook} that we quoted before deals with a reference form $\om$, which is semi-positive and big on a compact K\"ahler manifold. 
In the singular setting, after resolving the singularities, the proof of the properties which we recalled in Section~\ref{sec:prelim_MA_energy} can be reduced to the same case on the desingularization. 
Since the Monge--Amp\`ere measure of bounded functions puts no mass on the exceptional set, one can deduce the properties on the singular cases.
\end{rmk}

\subsection{$L^1$ metric geometry}
The finite energy class $\CE^1(X,\omega)$ can be endowed with the distance
\[
    d_1(u,v):=\E(u)+\E(v)-2\E\big(P_\om(u,v)\big)
\]
(see \cite[Thm. 2.1]{Darvas_2017}, \cite[Thm. B]{Dinezza_Guedj_2018}). 
Here, we denote by 
\[
    P_\om(u,v) := \Big(\sup \set{w\in \PSH(X,\om)}{w \leq \min(u,v)}\Big)^\ast
\]
the largest $\om$-psh function smaller than the minimum between $u,v$ (see \cite{Ross_WittNystrom_2014}). 
The metric topology on $\big(\CE^1(X,\om),d_1\big)$ is the coarsest refinement of the $L^1$-topology that makes the Monge--Amp\`ere energy continuous, and it is related to the stability of solutions of complex Monge--Amp\`ere equations (see \cite{BBEGZ_2019, Trusiani_2023}). Moreover, $\big(\CE^1(X,\om),d_1\big)$ is a complete metric space and any two elements $u,v\in \CE^1(X,\om)$ can be joined by a (weak) geodesic segment given as a solution of a homogeneous Monge--Amp\`ere equation \cite[Thm. 2.1]{Darvas_2017}, \cite[Thm. A]{Dinezza_Guedj_2018}. 

\subsection{Variational approach to K\"ahler--Einstein metrics}
We now assume that $X$ has Kawamata log terminal (klt) singularities.
Namely, for any desingularization $p: Y \to X$,  
\[
    K_Y = p^\ast K_X + \sum_{i} a_i E_i 
    \quad\text{with } a_i > -1
\]
where $E_i$ are irreducible components of the exceptional divisor $\Exc(p)$.

\begin{defn}\label{defn:adapted_measure}
Suppose that $K_X$ is $m$-Cartier for some $m \in \BN$. 
Let $h^m$ be a smooth metric on $m K_X$. 
Taking $\Om$ a local generator of $m K_X$, the adapted measure associated with $h^m$ is defined as 
\[
    \mu_h := \ii^{n^2} \lt(\frac{\Om \w \overline{\Om}}{|\Om|_{h^m}^2}\rt)^\frac{1}{m}.
\]
Note that this definition does not depend on the choice of $\Om$. 
\end{defn}

By the klt assumption, one has $\mu_h = f \om^n$ for some $f \in L^p(X,\om^n)$ with $p>1$ (cf. \cite[Sec.~6]{EGZ_2009}). 
Rescaling $h$, we may assume that $\mu_h$ is a probability measure. 

\smallskip
Now, we further assume that $X$ is $\BQ$-Fano and $\om\in c_1(X)$. 
The Ding functional 
\[
    \D:\CE^1(X,\om)\to \BR
\] 
associated with $\mu_h$ is 
\[
    \D(u) := -\E(u) - \log\lt(\int_X e^{-u} \dd\mu_h\rt).
\]
The Euler--Lagrange equation of $\D$ is the following complex Monge--Amp\`ere equation 
\begin{equation}\label{MAKE}
    \frac{1}{V}(\om + \ddc u)^n =  \frac{e^{-u} \mu_h}{\int_X e^{-u}\dd \mu_h}.
\end{equation}
Solutions in $\CE^1(X,\om)$ of (\ref{MAKE}) are actually genuine K\"ahler--Einstein metrics on the smooth locus, and the potentials are continuous on $X$ (see \cite[Thm.~A]{EGZ_2009} and \cite[Cor.~C]{CGZ_2013}).

\smallskip
We can now state the following analytic characterization on the existence of K\"ahler--Einstein metrics:

\begin{thm}[{\cite{Darvas_Rubinstein_2017, Dinezza_Guedj_2018}}]\label{Summ_Thm}
Suppose that $X$ is klt $\BQ$-Fano and $\Aut(X)$ is discrete. 
The following are equivalent:
\begin{enumerate}[label={(\roman*)}]
    \item\label{item:Summ_Thm_1} There exists a K\"ahler--Einstein metric;
    \item\label{item:Summ_Thm_2} The Ding functional is coercive, i.e. there exist $A \in (0,1)$ and $B \geq 0$ such that for all $u \in \CE_{\nmlz}^1(X,\om) := \set{v \in \CE^1(X,\om)}{\sup_X v = 0}$, 
    \[
        \D(u) \geq A(-\E(u)) - B.
    \]
\end{enumerate}
\end{thm}

A key result in proving Theorem \ref{Summ_Thm} is the following convexity theorem due to Berndtsson~\cite{Berndtsson_2015} (see also \cite{BBEGZ_2019} for the singular setting).
\begin{thm}[{\cite[Thm.~1.1]{Berndtsson_2015}, \cite[Lem.~4.6]{BBEGZ_2019}}]\label{thm:Berndtsson_cvx}
Let $[0,1]\ni t\to u_t\in \CE^1(X,\omega)$ be a (weak) geodesic segment. 
Then 
\[
    [0,1]\ni t\to -\log\Big(\int_Xe^{-u_t}\dd \mu_h\Big)
\] 
is convex, while $[0,1]\ni t\to \E(u_t)$ is affine. 
In particular, the Ding functional is geodesically convex.
\end{thm}

\section{Convergence of quasi-plurisubharmonic functions in families}\label{sec:conv_fami}
In this section, we define a notion of convergence of quasi-psh functions in families. 
Then we study Hartogs type properties with respect to this topology. 

\subsection{Setting and known facts}
In the sequel, we always assume that a family $\pi: \CX \to \BD$ fits in the following setup:

\begin{taggedbigset}{(GSN)}\label{sett:general_sett_normal_fibre}
Let $\CX$ be an $(n+1)$-dimensional, irreducible, and reduced complex K\"ahler space.
Let $\pi: \CX \to \BD$ denote a proper, surjective, holomorphic map such that each fibre $X_t := \pi^{-1}(t)$ is an irreducible, reduced, compact K\"ahler space for any $t \in \BD$. 
In addition, assume that $X_0$ is normal.

Let $\om$ be a K\"ahler metric on $\CX$. 
Namely, $\om$ is a K\"ahler metric $\om$ on $\CX^\reg$ such that given any local embedding $j: \CX \hookrightarrow \BC^N$, $\om$ extends smoothly to a K\"ahler metric defined on an open neighborhood of $\im(j)$.
For each $t \in \BD$, $\om$ induces a K\"ahler metric $\om_t$ on $X_t$ by restriction (i.e. $\om_t := \om_{|X_t}$).
\end{taggedbigset}

\begin{rmk}
The volume $V_t := \int_{X_t} \om_t^n$ is independent of $t \in \BD$ (cf. \cite[Lem.~2.2]{DGG2020}).
We shall denote it by $V$ in the following.
\end{rmk}

\begin{rmk}
From \cite[Thm.~12.2.1 (v)]{Grothendieck_EGA_4_3_1966}, normality is open on the base $\BD$ if the map $\pi$ is flat; namely, $X_t$ is normal for all $t$ sufficiently close to zero. 
On the other hand, if $X_t$ is normal for every $t \in \BD$, then so is $\CX$ by \cite[Cor.~5.12.7]{Grothendieck_EGA_4_2_1965}.
Therefore, up to shrinking $\BD$, we can assume that $\CX$ and $(X_t)_{t \in \BD}$ are normal. 
\end{rmk}

We provide the following example satisfying Setting~\ref{sett:general_sett_normal_fibre}: 
Consider 
\[
    \CX := \set{([x:y:z:w],t) \in \BP^3 \times \BC}{xyz + w^3 + tx^3 = 0} \subset \BP^3 \times \BC
\]
and $\pi: \CX \to \BC$ given by the restriction of the projection $\pr_2 : \BP^3 \times \BC \to \BC$.
The central fiber $X_0$ has three isolated singularities, and for $t \neq 0$, the general fiber $X_t$ has only two singular points. 
The fibres are normal as they are hypersurfaces in $\BP^3$ and have a singular set whose codimension is at least $2$. 
A K\"ahler metric $\om$ on $\CX$ can easily be obtained by the restriction of the product of the Fubini--Study metric on $\BP^3$ and the Euclidean metric on $\BC$. 

\smallskip
We recall uniform integrability results of Skoda--Zeriahi and Sup-$L^1$ comparison of quasi-psh functions in families from \cite[Thm.~2.9]{DGG2020} and \cite[Cor.~4.8]{Ou_2022}: 

\begin{thm}\label{thm:SL_and_Skoda_in_family}
In Setting~\ref{sett:general_sett_normal_fibre}, there exist constants $C_{SL} > 0, \af > 0$, and $C_\af > 0$ such that 
\[
    \sup_{X_t} \psi_t - C_{SL} \leq \frac{1}{V} \int_{X_t} \psi_t \om_t^n,
    \quad\text{and} \quad
    \int_{X_t} e^{- \af (\psi_t - \sup_{X_t} \psi_t)} \om_t^n \leq C_\af
\]
for all $t \in \bBD_{1/2}$ and for every $\psi_t \in \PSH(X_t, \om_t)$.
\end{thm}

\subsection{Definition of the convergence in families}\label{subsec:conv_in_fami}
Now, we are going to introduce a notion of convergence of quasi-psh functions on different fibres towards a quasi-psh function on the central fibre.
A similar concept has shown up in \cite[Prop.~6.6]{Bakker_Guenancia_Lehn_2022} on a locally trivial family.
It is natural to look for an analogous convergence in more general settings. 

\smallskip
We first fix some notation. 
Suppose that $\CZ$ is the singular set of the map $\pi$.
For each point $x \in X_{0}^\reg$, up to shrinking $\BD$, there are a chart $U_0 \Subset X_0^\reg$ containing $x$, an open subset $\CU \Subset \CX \setminus \CZ$ with $\CU \cap X_0 = U_0$, and an isomorphism $F: U_0 \times \BD \to \CU$ such that the diagram
\[
\begin{tikzcd}
    F: U_0 \times \BD \ar[rr, "\sim"] \ar[rd, "\pr_2"']&& \CU \ar[ld, "\pi"]\\
    &\BD&
\end{tikzcd}
\]
commutes and $F_{|U_0} = \Id_{U_0}$.
We denote by 
\[
    F_t: U_0 \overset{\sim}{\longrightarrow} U_t := \CU \cap X_t
\] 
the isomorphism induced by $F$. 

\begin{defn}\label{defn:conv_in_family}
For all $k \in \BN$, let $u_{t_k}$ be a $\om_{t_k}$-psh function on $X_{t_k}$ and $t_k \to 0$ as $k \to +\infty$.
We say that the sequence $(u_{t_k})_k$ converges to $u_{0} \in \PSH(X_0, \om_{0})$ in $L^1$ (resp. $\CC^0$, $\CC^\infty$) if for all data $(U_0, F, \CU)$ as above, $u_{t_k} \circ F_{t_k}$ converges to $u_{0}$ in $L^1(U_0)$ (resp. $\CC^0(U_0)$, $\CC^\infty(U_0)$).
\end{defn}

In the sequel, we call such a sense of $L^1$ (resp. $\CC^0$, $\CC^\infty$) convergence the {\it convergence in families} (resp. {\it $\CC^0$, $\CC^\infty$-convergence in families}) or we say a sequence {\it converging in the family sense} (resp. {\it $\CC^0$, $\CC^\infty$-converging in the family sense}). 

\smallskip
To check Definition~\ref{defn:conv_in_family} is well-defined, we recall a consequence of Hartogs lemma (see e.g. \cite[Thm.~1.46]{GZbook}):
\begin{lem}\label{lem:L1_conv_limsup}
Let $\Om$ be a domain in $\BC^n$. 
If $(u_k)_{k \in \BN} \in \PSH(\Om)^{\BN}$ converges to a psh function $u$ in $L^1_{\loc}(\Om)$, then for any sequence $(x_k)_k$ converging to $x \in \Om$,
\[
    u(x) \geq \limsup_{k \to +\infty} u_k (x_k).
\]
\end{lem}

\begin{rmk}\label{rmk:limit_well_defined}
Using Lemma~\ref{lem:L1_conv_limsup}, one can prove that the convergence from Definition~\ref{defn:conv_in_family} does not depend on the choice of charts and local isomorphisms.
To see this, we let $u_1$ (resp. $u_2$) be the $L^1$-limit of $\phi_k := u_{t_k} \circ F_{t_k}$ (resp. $\psi_k := u_{t_k} \circ G_{t_k}$) on $U_0$ (resp. $V_0$).
Fix $x \in U_0 \cap V_0$ and consider $x_k := (F_{t_k})^{-1} \circ G_{t_k}(x)$ which converges to $x$ as $k \to +\infty$.
Using Lemma~\ref{lem:L1_conv_limsup}, we obtain 
\[
    u_1(x) \geq \limsup_{k \to +\infty} \phi_k(x_k) 
    = \limsup_{k \to +\infty} \psi_k (x).
\]
According to the classical Hartogs lemma (cf. \cite[Thm.~1.46]{GZbook}), 
one has 
\[
    u_2(x) \geq \limsup_{k \to +\infty} \psi_k(x)
\] 
and the equality holds almost everywhere; 
hence $u_1 \geq u_2$ almost everywhere on $U_0 \cap V_0$ and so is everywhere.
Then exchanging $u_1$ and $u_2$ in the argument, one can infer $u_2 \geq u_1$ on $U_0 \cap V_0$.
This completes the proof.
\end{rmk}

\begin{nota}\label{nota}
In the sequel, we shall simply denote $t_k$ by $k$, as well as $X_k := X_{t_k}$ and $\om_k := \om_{t_k}$.
We often write 
\[
    (u_k)_{k \in \BN} \in \PSH_{\fibre}(\CX, \om)
\] 
for a sequence such that $u_k \in \PSH(X_k, \om_k)$ for each $k$.

\smallskip

On the other hand, let $\CZ$ be the singular locus of $\pi$. 
For an open subset $\CU \Subset \CX \setminus \CZ$, we say that the data $(B_\af, F_\af, \CB_\af)_\af$ is a locally trivial finite cover of $\CU$ if 
\begin{itemize}
    \item the set of indices is finite;
    \item each $B_\af$ is a relatively compact open subset on $X_0^\reg$ and $B_\af$ is isomorphic to a ball of radius $r_\af$ in $\BC^n$; we also set $B_\af = B^{X_0} (x_\af, r_\af)$ where $x_\af$ is the center; 
    \item up to shrinking $\BD$, $F_\af$ is an isomorphism between $B_\af \times \BD$ and $\CB_\af$, 
    \[
    \begin{tikzcd}
        F_\af: B_\af \times \BD \ar[rr, "\sim"] \ar[rd, "\pr_2"'] && \CB_\af \ar[ld, "\pi"]\\
        &\BD&
    \end{tikzcd}
    \] is a commutative diagram, 
    \[
        U_0 := \CU \cap X_0 \subset \bigcup_\af \lt(\frac{1}{2} B_\af\rt) 
        \quad\text{and}\quad 
        \CU \subset \bigcup_\af \lt(\frac{1}{2} \CB_\af\rt),
    \] 
    where $r \CB_\af = F_\af(rB_\af \times \BD)$ and $r B_\af = B^{X_0}(x_0, r \cdot r_\af)$.
\end{itemize}
We denote by $F_{\af,t}: B_\af \overset{\sim}{\longrightarrow} B_{\af,t} := \CB_\af \cap X_t$ the isomorphism induced by $F_\af$.
\end{nota}

\subsection{Hartogs type results}
In this section, we study Hartogs type results and the continuity of the supremum along the convergence in families.

\begin{prop}\label{prop:Hartogs_conv_subseq_family}
The following hold:
\begin{enumerate}[label={(\roman*)}]
    \item\label{item:hartogs_1} Let $(u_k)_k \in \PSH_{\fibre}(\CX, \om)$ be a sequence with $t_k \to 0$ as $k \to +\infty$. 
    If $(\sup_{X_k} u_k)_k$ is uniformly bounded, there is a subsequence converging to a function $u_0 \in \PSH(X_0, \om_0)$.
    \item\label{item:hartogs_2} If $(u_k)_k \in \PSH_{\fibre}(\CX, \om)$ converges to $u_0 \in \PSH(X_0, \om_0)$, then $(\sup_{X_k} u_k)_k$ is bounded and the following lower semi-continuity holds:
    \[
        \sup_{X_0} u_0 \leq \liminf_{k \to +\infty} \sup_{X_k} u_k.
    \]
\end{enumerate}
\end{prop}

\begin{proof}
We first prove \ref{item:hartogs_1}.
Without loss of generality, one may assume that $(u_k)_k$ are negative functions.
From Theorem~\ref{thm:SL_and_Skoda_in_family}, the sequence $\lt(\int_{X_k} u_k \om_k^n\rt)_k$ is bounded.
One can find countably many double covers $(U_{j,0})_{j \in \BN}$ and $(V_{j,0})_{j \in \BN}$ of $X_0^\reg$ with $V_{j,0} \Subset U_{j,0}$ for each $j \in \BN$ and families of biholomorphisms $(F_t^{U_{j,0}})_j$.
Since $\lt(\int_{X_k} u_k \om_k^n\rt)_k$ is bounded, for each $j \in \BN$, $\lt(\int_{U_{j,0}} (F_k^{U_{j,0}})^\ast u_k \dd V_{\text{eucl}}\rt)_k$ is also bounded.
Applying Hartogs lemma, there exists a subsequence of $(u_k)_k$ such that $((F_k^{U_{1,0}})^\ast u_k)_k$ converges to a function $u_0 \in \PSH(V_{1,0}, \om_0)$ along the subsequence.
Abusing notation, we still denote the above subsequence by $(u_k)_k$.
Take $V_{2,0}$ with $V_{2,0} \cap V_{1,0} \neq \emptyset$.
Next, by the same reason, the family $(F_k^{U_{2,0}})^\ast u_k$ is also relatively compact in $L^1_\loc(U_{2,j})$.
Then one can again extract a convergent subsequence, and the limit coincides with $u_0$ on $V_{1,0} \cap V_{2,0}$ by Remark~\ref{rmk:limit_well_defined}.
Iterating the argument and taking the diagonal subsequence, we obtain a convergent subsequence in the sense of Definition~\ref{defn:conv_in_family}.
The limit function $u_0 \in \PSH(X_0^\reg, \om_0)$ extends to a function $u_0 \in \PSH(X_0, \om_0)$ by normality of $X_0$ and Grauert--Remmert's theorem \cite{Grauert_Remmert_1956}.
In particular, $u_0$ is not identically $-\infty$.

\smallskip
We now deal with \ref{item:hartogs_2}.
On each open chart $U_0 \Subset X_0^\reg$ with a family of biholomorphisms $F_t$ to a chart $U_t$ of nearby fibre $X_t$ and $F_0 = \Id_{U_0}$, the sequence $(F_k^\ast u_k)_k$ converges to $u_0$ in $L^1(U_0)$. 
Therefore, for a.e. $x \in U_0$, 
\[
    u_0(x) = \liminf_{k \to +\infty} F_k^\ast u_k(x) 
    \leq \liminf_{k \to +\infty} \sup_{X_k} u_k.
\]
It implies that ${u_0}_{|X_0^\reg} \leq \liminf_{k \to +\infty} \sup_{X_k} u_k$ by plurisubharmonicity and thus 
\[
    \sup_{X_0} u_0 = \sup_{z \in X_0} \lt(\limsup_{X_0^\reg \ni x \to z} u_0(x)\rt) 
    \leq \liminf_{k \to +\infty} \sup_{X_k} u_k.
\]

\smallskip
To control $(\sup_{X_k} u_k)_k$ from above, we argue by contradiction. 
Assume that there is a sequence $(u_k)_k$ converging in families and $(\sup_{X_k} u_k)_k$ is not bounded from above.
After extracting and relabeling, one can assume that $(\sup_{X_k} u_k)_k$ increases to $+\infty$. 
Then for each $j$ fixed, we have the sequence $(u_k - \sup_{X_j} u_j)_k$ converges to $u_0 - \sup_{X_j} u_j$ in families.  
On the other hand, by Proposition~\ref{prop:Hartogs_conv_subseq_family}, the sequence $(u_k - \sup_{X_k} u_k)_k$ contains a convergent subsequence to $u \in \PSH(X_0, \om_0)$ which is not identically $-\infty$. 
Since $u_k - \sup_{X_j} u_j \geq u_k - \sup_{X_k} u_k$ for all $k \geq j$, we have $u_0 - \sup_{X_j} u_j \geq u$ for all $j \in \BD$.  
Taking $j \to +\infty$, this yields a contradiction.
\end{proof}

\begin{rmk}
If $\pi: \CX \to \BD$ is trivial (i.e. there is a biholomorphism $f: \CX \overset{\sim}{\longrightarrow} X_0 \times \BD$ satisfying $\pi = f \circ \pr_2$) and $u_k \in \PSH(X_k, \om_k)$ is a sequence converging to $u_0 \in \PSH(X_0, \om_0)$, the convergence defined in Definition~\ref{defn:conv_in_family} is the usual $L^1$-convergence on $X_0$.

\smallskip
Indeed, for all $\vep > 0$, one can find an open subset $\CZ_{0,\vep} \Subset \CW_{0,\vep} \subset X_0$ containing $X_0^\sing$ and satisfying $\Vol_{\om_0}(\CW_{0,\vep}) < \vep/2$. 
We choose finite open charts $(U_{j,0})_j$ of $X_0^\reg$ which covers $\CX \setminus \CZ_\vep$. 
By triviality of the family, one can identify these data to any $X_t$. 
Therefore, up to shrinking $\BD$, we obtain an open set $\CW_{t,\vep}$ near $X_t^\sing$ such that $\Vol_{\om_t}(\CW_{t,\vep}) < \vep$ and a locally trivial finite cover $(B_\af, F_\af, \CB_\af)_\af$ of $\CX \setminus \overline{\CZ_{\vep}}$.
Since $(u_k)_k$ is a sequence which converges to $u_0$ in the family sense, $u_k \circ F_{\af,k}$ converges to $u_0$ in $L^1(B_\af)$.
By Proposition~\ref{prop:Hartogs_conv_subseq_family}~\ref{item:hartogs_2}, we may assume that $u_k$ are negative and $-D$ is a lower bound of $(\sup_{X_k} u_k)_k$ and $\sup_{X_0} u_0$.
For every $t$ close to zero, one can find a uniform constant $C>0$ such that $\om_0$ and $\om_{t}$ are quasi-isometric (i.e. $C^{-1} \om_0 \leq \om_t \leq C \om_0$).

\smallskip
Near singularities, applying H\"older's inequality and Theorem~\ref{thm:SL_and_Skoda_in_family}, we have
{\footnotesize
\begin{align*}
    \int_{\CW_{0,\vep}} \abs{u_k \circ F_k - u_0} \om_0^n 
    &\leq \int_{\CW_{t_k,\vep}} |u_k| \om_0^n + \int_{\CW_{0\vep}} |u_0| \om_0^n\\
    &\leq C\Vol_{\om_k}(\CW_{t_k,\vep})^{1/2} \lt(\int_{\CW_{t_k,\vep}} |u_k|^2 \om_k^n\rt)^{1/2} + \Vol_{\om_0}(\CW_{0,\vep})^{1/2}\lt(\int_{\CW_{0, \vep}} |u_0|^2 \om_0^n\rt)^{1/2}\\
    &\leq (C+1) \vep^{1/2} \frac{\af}{\sqrt{2}} e^{\af D/2} \lt[ \lt( \int_{X_k} e^{-\af \lt( u_k - \sup_{X_k} u_k \rt)} \om_k^n \rt)^{1/2} + \lt(\int_{X_0} e^{-\af\lt(u_0 - \sup_{X_0} u_0\rt)} \om_0^n\rt)^{1/2} \rt]\\
    &\leq (C+1) \af e^{\af D/2} C_\af^{1/2} \vep^{1/2}.
\end{align*}}%
Thus, in trivial families, we show that if $(u_k)_k$ converges to $u_0$ in the family sense, then the convergence is the usual $L^1$-convergence. 
The reverse implication is obvious.
\end{rmk}

Therefore, the convergence in families is a natural extension of the usual $L^1$-convergence. 
Next, we prove that along a sequence $(u_k)_k$ converging in families, the associated currents $(\ddc u_k \w [X_k])_k$ also converges to $\ddc u_0 \w [X_0]$ in the sense of currents on the total space $\CX$. 

\begin{prop}\label{prop:hartogs_currents}
The following hold:
\begin{enumerate}[label={(\roman*)}]
    \item\label{item:hartogs_currents_1} 
    If $(u_k)_k \in \PSH_{\fibre}(\CX, \om)$ converges to $u_0 \in \PSH(X_0 , \om_0)$, then $\ddc u_k \w [X_k]$ converges to $\ddc u_0 \w [X_0]$ in the sense of currents on $\CX$.
    \item\label{item:hartogs_currents_2}
    Let $(u_k)_k \in \PSH_{\fibre}(\CX,\om)$ with $t_k \to 0$ as $k \to +\infty$ and $\sup_{X_k} u_k = 0$. 
    Assume that there exists $u_0 \in \PSH(X_0,\om_0)$ with $\sup_{X_0} u_0 = 0$ and $(\ddc u_k \w [X_k])_k$ converges to $\ddc u_0 \w [X_0]$ in the sense of currents. 
    Then up to extracting a subsequence, $(u_k)_k$ converges to $u_0 + c$ for some constant $c < 0$. 
\end{enumerate}
\end{prop}

\begin{proof}
We first deal with part \ref{item:hartogs_currents_1}.
Without loss of generality, we may assume that $(u_k)_k$ and $u_0$ are negative.
Let $D \geq 0$ be a constant such that $\sup_{X_k} u_k \geq -D$ for all $k \in \BN$ and $\sup_{X_{0}} u_{0} \geq -D$.
By the Taylor expansion of $x \mapsto e^x$ and Theorem~\ref{thm:SL_and_Skoda_in_family}, we have
\begin{equation}\label{eq:bdd_Lp}
    \int_{X_k} \abs{u_k}^2 \om_k^n 
    \leq \frac{2}{\af^2} e^{-\af \sup_{X_k} u_k} \int_{X_k} e^{-\af \lt(u_k - \sup_{X_k} u_k\rt)} \om_k^n 
    \leq \frac{2}{\af^2} e^{\af D} C_\af=:A.
\end{equation}
Similarly $\int_{X_0}\abs{u_0}^2\om_0^n\leq A$.

For $\vep > 0$, we fix open sets $\CW_{\vep/2} \Subset \CW_\vep$ in $\CX$ both containing $X_0^\sing$ such that 
\begin{equation}\label{eq:vol_small_open_near_sing}
    \Vol_{\om_{t}}(X_t \cap \CW_\vep) < \vep
\end{equation}
for all $t$ sufficiently close to $0$.
We set 
\[
    \CU_\vep := \CX \setminus \overline{\CW_{\vep/2}}.
\]
Up to shrinking $\BD$, we take $(B_\bt, F_\bt, \CB_\bt)_{\bt \in J}$ a locally trivial finite cover of $\CU_\vep$. 
One can find $(\chi_\bt)_{\bt \in J} \cup \{\chi_\CW\}$, a partition of unity associated to $\bigcup_{\bt}\CB_\bt \cup \CW_\vep$ over a small disk, i.e.:
\begin{itemize}
    \item $\chi_\CW$ is supported in $\CW_\vep$ and for each $\bt$, $\chi_\bt$ is supported in $\CB_\bt$;
    \item $\chi_\CW + \sum_{\bt \in J} \chi_\bt \equiv 1$ on $\pi^{-1}(\bBD_r)$ for some $r>0$ sufficiently small.
\end{itemize} 
Fix a test $(n-1,n-1)$-form $\eta$, and write $\eta_t:=\eta_{|X_t}$, $\chi_{\bt,t} := {\chi_\bt}_{|X_t}$ and $\chi_{\CW,t} := {\chi_\CW}_{|X_t}$. Then we have the following estimate: 
\begin{align*}
    &\abs{\iprod{\eta}{\ddc u_k \w [X_k]} - \iprod{\eta}{\ddc u_0 \w [X_0]}}  
    = \abs{\int_{\CX} \eta \w \lt(\ddc u_k \w [X_k] - \ddc u_0 \w [X_0]\rt)}\\
    &= \abs{\int_{\CX} \ddc \eta \w \lt(u_k [X_k] - u_0 [X_0]\rt)}\\
    &\leq \abs{\sum_{\bt \in J} \int_{\CB_\bt} \chi_\bt \ddc \eta \w \lt(u_k [X_k] - u_0 [X_0]\rt)}
    + \abs{\int_{\CW_\vep} \chi_\CW \ddc \eta \w \lt(u_k [X_k] - u_0 [X_0]\rt)}\\
    &\leq \underbrace{\abs{\sum_{\bt \in J} \lt(\int_{B_{\bt,k}} u_k \chi_{\bt,k} \ddc_k \eta_k - \int_{B_{\bt}} u_0 \chi_{\bt,0} \ddc_0 \eta_0 \rt) }}_{=: \RN{1}}
    + C_\eta \underbrace{\lt(\int_{X_k \cap \CW_\vep} |u_k| \om_k^n + \int_{X_0 \cap \CW_\vep} |u_0| \om_0^n\rt)}_{=: \RN{2}}.
\end{align*}
We first control the first term $\RN{1}$.
It follows directly from the construction that as $k \to +\infty$,
\begin{equation}\label{eq:local_test_conv}
    \frac{F_{\bt,k}^\ast(\chi_{\bt,k} \ddc_k \eta_k)}{\om_0^n} 
    \to \frac{\chi_{\bt,0} \ddc_0 \eta_0}{\om_0^n}
\end{equation}
uniformly on $B_\bt$ for any $\bt\in J$. Since $\abs{\ddc_0 \eta_{0}}\leq C_\eta \om_0^n$ for a positive constant $C_\eta>0$, the uniform convergence \eqref{eq:local_test_conv} yields
\begin{align*}
    \limsup_{k \to +\infty} \RN{1} 
    &\leq \limsup_{k \to +\infty} \sum_{\bt \in J} \int_{B_\bt} \abs{(u_k \circ F_{\bt,k}) \frac{F^\ast_{\bt,k} (\chi_{\bt,k} \ddc_k \eta_k)}{\om_0^n} - u_0 \frac{\chi_{\bt,0} \ddc_0 \eta_0}{\om_0^n}} \om_0^n\\ 
    &\leq \limsup_{k \to +\infty} \sum_{\bt \in J} \int_{B_\bt} \abs{u_k \circ F_{\bt,k} - u_0} \abs{\frac{F^\ast_{\bt,k} (\chi_{\bt,k} \ddc_k \eta_k)}{\om_0^n}} \om_0^n \\
    &\qquad+ \limsup_{k \to +\infty} \sum_{\bt \in J} \int_{B_\bt} \abs{u_0} \abs{\frac{F^\ast_{\bt,k}(\chi_{\bt,k} \ddc_k \eta_k)}{\om_0^n} - \frac{\chi_{\bt,0} \ddc_0 \eta_0}{\om_0^n}} \om_0^n\\
    &\leq 2C_\eta \limsup_{k \to +\infty} \sum_{\bt \in J} \int_{B_\bt} \abs{u_k \circ F_{\bt,k} - u_0} \om_0^n = 0,
\end{align*}
where the last convergence follows from the definition of convergence in families of $u_k$ to $u_0$.

\smallskip
By H\"older's inequality, \eqref{eq:bdd_Lp} and \eqref{eq:vol_small_open_near_sing}, one obtains a control on the second part $\RN{2}$:
{\small
\begin{equation*}
    \RN{2} \leq \lt(\int_{X_k \cap \CW_\vep} |u_k|^2 \om_k^n\rt)^{1/2} \Vol_{\om_k}^{1/2}(X_k \cap \CW_\vep) + \lt(\int_{X_{0} \cap \CW_\vep} |u_{0}|^2 \om_{0}^n\rt)^{1/2} \Vol_{\om_{0}}^{1/2}(X_{0} \cap \CW_\vep) \leq 2 A \vep^{1/2}.
\end{equation*}}
All in all, we obtain that
\[
    \limsup_{k \to +\infty} \abs{\iprod{\eta}{\ddc u_k \w [X_k]} - \iprod{\eta}{\ddc u_0 \w [X_0]}}
    \leq 2A\vep^{1/2}.
\]
Letting $\vep\to 0$ we deduce that $\ddc u_k \w [X_k]$ converges to $\ddc u_0 \w [X_0]$ in the sense of currents.

\smallskip

We now show part~\ref{item:hartogs_currents_2}.
By Proposition~\ref{prop:Hartogs_conv_subseq_family}, one can extract a subsequence converging in families towards $v_0 \in \PSH(X_0, \om_0)$ and $\sup_{X_0} v_0 \leq 0$.
We have $\ddc u_0 \w [X_0] = \ddc v_0 \w [X_0]$. 
Then $u_0 - v_0$ is a pluriharmonic distribution on $X_0$; hence $u_0 - v_0$ is constant, and we have $u_0 = v_0 - \sup_{X_0} v_0$ by the normalization of $u_0$.
\end{proof}

In Proposition~\ref{prop:Hartogs_conv_subseq_family}~\ref{item:hartogs_2}, if $(u_k)_k$ converges to $u_0$ in the family sense, we obtain an upper bound of $\sup_{X_0} u_0$ in terms of $(\sup_{X_k} u_k)_k$.  
One may wonder whether there is a good formulation for the lower bound of $\sup_{X_0} u_0$ in terms of the $(\sup_{X_k} u_k)_k$. 
We next prove the convergence of $(\int_{X_k} u_k \om_k^n)_k$ towards $\int_{X_0} u_0 \om_0^n$ and derive a lower bound for $\sup_{X_0} u_0$.

\begin{lem}\label{lem:L1_integral_converge}
If $(u_k)_k \in \PSH_{\fibre}(\CX, \om)$ converges to $u_0 \in \PSH(X_0 , \om_0)$, then 
\[
    \int_{X_k} u_k \om_k^n 
    \xrightarrow[k \to +\infty]{} \int_{X_0} u_0 \om_0^n.
\]
Moreover, one has
\[
    \sup_{X_0} u_0 \geq \limsup_{k \to +\infty} \lt(\sup_{X_k} u_k\rt) - C_{SL},
\]
where $C_{SL}$ is a uniform constant in the sup-$L^1$ comparison in families (see Theorem~\ref{thm:SL_and_Skoda_in_family}). 
\end{lem}

\begin{proof}
The proof follows the same idea as in Proposition~\ref{prop:hartogs_currents}.
We use the same background data as in the proof of Proposition~\ref{prop:hartogs_currents}. 
By H\"older's inequality, we get
\begin{align*}
    &\abs{\int_{X_k} u_k \om_k^n - \int_{X_0} u_0 \om_0^n}\\
    &\leq \abs{\sum_{\af \in J} \lt(\int_{B_{\af,k}} \chi_{\af,k} u_k \om_k^n - \int_{B_\af} \chi_{\af,0} u_0 \om_0^n\rt)} 
    + \abs{\int_{X_k \cap \CW_\vep} \chi_{\CW,k} u_k \om_k^n} 
    + \abs{\int_{X_0 \cap \CW_\vep} \chi_{\CW,0} u_0 \om_0^n}\\
    &\leq \sum_{\af \in J} \int_{B_\af} \abs{\frac{F_{\af,k}^\ast (\chi_{\af,k} u_k \om_k^n)}{\om_0^n} - \chi_{\af,0} u_0} \om_0^n 
    + 2 \lt(\frac{2}{\af^2} e^{\af D} C_\af\rt) \vep^{1/2}.
\end{align*}
Since $(F_{\af,k}^\ast u_k)_k$ converges to $u_0$ in $L^1(B_\af)$ and $\lt(\frac{F_{\af,k}^\ast(\chi_{\af,k} \om_k^n)}{\om_0^n}\rt)_k$ converges smoothly to $\chi_{\af,0}$ on $B_\af$, we have 
\[
    \int_{B_\af} \abs{\frac{F_{\af,k}^\ast (\chi_{\af,k} u_k \om_k^n)}{\om_0^n} - \chi_{\af,0} u_0} \om_0^n 
    \xrightarrow[k \to +\infty]{} 0;
\]
hence, $\lim_{k \to +\infty} \int_{X_k} u_k \om_k^n = \int_{X_0} u_0 \om_0^n$.

\smallskip
By Theorem \ref{thm:SL_and_Skoda_in_family}, there is a uniform constant $C_{SL} > 0$ such that
\[
    \frac{1}{V} \int_{X_k} u_k \om_k^n \geq \sup_{X_k} u_k - C_{SL} \quad\text{for all $k$}.
\] 
Letting $k \to + \infty$, one can conclude that
\[
    \limsup_{k \to + \infty} \lt(\sup_{X_k} u_k\rt) - C_{SL}
    \leq \lim_{k \to +\infty} \frac{1}{V} \int_{X_k} u_k \om_k^n 
    = \frac{1}{V} \int_{X_0} u_0 \om_0^n 
    \leq \sup_{X_0} u_0.
\]
\end{proof}

We conjecture that the supremum is actually continuous along a sequence converging in families. 
Using the maximum principle, we show this is the case when $X_0$ has 
isolated singularities:

\begin{prop}\label{prop:usc_sup_conv}
Suppose that $(u_k)_k \in \PSH_{\fibre}(\CX, \om)$ is a sequence converging to $u_0 \in \PSH(X_0, \om_0)$. 
If the central fibre $X_0$ has isolated singularities, then 
\[
    \sup_{X_0} u_0 = \limsup_{k \to +\infty} \sup_{X_k} u_k.
\]
\end{prop}

\begin{proof}
When $\pi: \CX \to \BD$ is smooth, upper semi-continuity follows directly by Hartogs lemma.
Now, we assume that $X_0$ has isolated singularities.
Let $\CZ \subset \CX$ be the singular locus of $\pi$.
Since $t \to \dim X_t \cap \CZ$ is upper semi-continuous, the nearby fibres have only isolated singularities as well.
Let $x_k \in X_k$ be a point where $u_k$ achieves its maximum.
After extracting a subsequence, we assume that $(x_k)_k$ converges to a point $x_0 \in X_0$.
If $x_0 \in X_0^\reg$, the proof is the same as in the smooth case. 
Otherwise, we embed $\CX \xhookrightarrow[\loc.]{} \BC^N$ near $x_0 \in X_0^\sing$ and $x_0$ coincides with the origin under this local embedding.
Let $\BB_r \subset \BC^N$ be the Euclidean ball of radius $r>0$ centered at the origin and let $g$ be a smooth local potential of $\om$. 
Denote by $g_t$ the restriction of $g$ on $X_t$.
By the maximum principle, we have
\[
    \sup_{\pl\BB_r \cap X_k} \lt(g_k + u_k\rt)
    = \sup_{\BB_r \cap X_k} \lt(g_k + u_k\rt)
    \geq \lt(g_k + u_k\rt)(x_k)
    = g_k(x_k) + \sup_{X_k} u_k.
\]
Let $(y_k)_k$ be a sequence of points such that $y_k \in \pl\BB_r \cap X_k$ and 
\[
    (g_k + u_k) (y_k) = \sup_{\pl\BB_r \cap X_k} (g_k + u_k).
\]
One can extract a subsequence such that $(y_k)_k$ converges to $y_0 \in \pl\BB_r \cap X_0$.
By assumption, $y_0$ belongs to $\pl\BB_r \cap X_0 \subset X_0^\reg$; hence one can find an open chart $U_0 \Subset X_0^\reg$ which contains $y_0$ and a family of biholomorphisms $F_t: U_0 \to U_t \Subset X_t^\reg$ for all $t \in \BD$ close to $0$ and $F_0 = \Id_{U_0}$.
Using Lemma~\ref{lem:L1_conv_limsup}, we have
\begin{align*}
    \limsup_{k \to +\infty} (g_k + u_k) (y_k)
    &= \limsup_{k \to +\infty} \lt((F_k)^\ast (g_k + u_k)\rt) \lt((F_k)^{-1}(y_k)\rt)\\
    &\leq (g_0 + u_0)(y_0) 
    \leq g_0(y_0) + \sup_{X_0} u_0.
\end{align*}
Therefore, we obtain
\begin{align*}
    g_0(x_0) + \limsup_{k \to +\infty} \sup_{X_k} u_k
    &= \limsup_{k \to +\infty} \lt(g_k(x_k) + \sup_{X_k} u_k\rt)\\
    &\leq g_0(y_0) + \sup_{X_0} u_0
    \leq \sup_{\pl\BB_r \cap X_0} g_0 + \sup_{X_0} u_0.
\end{align*}
The upper semi-continuity follows by taking $r \to 0$ and the continuity of $g$.
Combining with Proposition~\ref{prop:Hartogs_conv_subseq_family}~\ref{item:hartogs_2}, one has the continuity of supremum.
\end{proof}

\section{Upper semi-continuity of Monge--Amp\`ere energies in families}\label{sec:usc_energy}
In this section, we establish Proposition~\ref{bigprop:usc_energy}, the upper semi-continuity of the energies with respect to the convergence in families.

\begin{prop}\label{prop:usc_energy_family}
Let $\pi: \CX \to \BD$ be a family that fulfills Setting~\ref{sett:general_sett_normal_fibre} and let $\om$ be a K\"ahler metric induced by the curvature of a positive $\BQ$-line bundle $\CL$ over $\CX$.
If a sequence $(u_k)_k \in \PSH_{\fibre}(\CX, \om)$ converges to $u_0 \in \PSH(X_0, \om_0)$ in the family sense, then 
\[
    \limsup_{k \to +\infty} \E_k (u_k) \leq \E_0 (u_0).
\]
\end{prop}

We conjecture that the upper semi-continuity of the Monge--Amp\`ere energy holds without additional assumptions on the K\"ahler metric $\om$. 

\smallskip
Define 
\[
    \CE_{\fibre/K, C}^1(\CX,\om) := \set{u_t \in \CE^1(X_t,\om_t)}{\E_t(u_t) \geq -C, u_t \leq 0, \text{ and } t \in K \subset \BD}.
\]
Combining Propositions~\ref{prop:Hartogs_conv_subseq_family} and \ref{prop:usc_energy_family}, we obtain the following compactness result:

\begin{cor}
Let $\pi: \CX \to \BD$ be a family satisfying Setting~\ref{sett:general_sett_normal_fibre}. 
Let $\om$ be a K\"ahler metric induced by the curvature of a positive $\BQ$-line bundle.
For all compact subsets $K \subset \BD$, the space $\CE^1_{\fibre/K, C}(\CX, \om)$ is compact with respect to the convergence in families. 
\end{cor}

\subsection{Preparations}
In this section, we review a few results used in the proof of Proposition~\ref{prop:usc_energy_family} with some modifications in the family setting. 
We first recall a slightly revised version of H\"ormander's $L^2$-estimates for the $\db$-operator (see e.g. \cite[Thm 5.1.1]{Berndtsson_Notes}).

\begin{lem}\label{lem:Hormander_L2}
Let $(X,\om)$ be a projective manifold. 
Let $E$ be a holomorphic line bundle equipped with a possibly singular hermitian metric $h_E e^{-\phi}$ where $h_E$ is a smooth hermitian metric, and $\phi$ is quasi-psh. 
Suppose that 
\[
    \ii\Ta_{h_E}(E) + \ddc \phi \geq \vep \om
    \quad \text{on an open set } U \subset X
\]
for some $\vep > 0$. 
Then for any $q \geq 1$ and for any $f \in \Gm(X, \bigwedge^{n,q} T_X^\ast \otimes E)$ with $\db f = 0$ and $\supp(f) \subset U$, there exists a section $g \in \Gm(X, \bigwedge^{n,q-1} T_X^\ast \otimes E)$ such that $\db g = f$ and 
\[
    \int_X |g|_{h_E, \om}^2 e^{-\phi} \om^n 
    \leq \frac{1}{\vep q} \int_X |f|_{h_E, \om}^2 e^{-\phi} \om^n.
\]
\end{lem}

Then we prove a uniform Ohsawa--Takegoshi theorem \cite{Ohsawa_Takegoshi_1987} on a relatively compact open subset inside the smooth locus.

\begin{lem}\label{lem:OT_smooth_loc}
Suppose that $\pi: \CX \to \BD$ fulfills Setting~\ref{sett:general_sett_normal_fibre} and $\CL$ is a positive line bundle over $\CX$. 
Let $\om = \ii\Ta_h(\CL)$ be a K\"ahler metric in $c_1(\CL)$.
Let $\CZ$ be the singular locus of $\pi$. 
Fix $\vep > 0$.
On any relatively compact open subset $\CU \Subset \CX \setminus \CZ$, there exist constants $m_0 = m_0(\vep, \CU, \CX, \om)$ and $C_{\OT} = C_{\OT}(\vep, \CU, \CX, \om)$ such that, up to shrinking $\BD$, for any
\begin{itemize}
    \item $m \geq m_0$,
    \item $x \in \CU \setminus X_0$,
    \item $0 \neq v \in {m L_\tau}_{|x}$ where $\tau = \pi(x)$,
    \item $u \in \PSH(X_\tau, \om_\tau)$ with $\om_\tau + \ddc_\tau u \geq \vep \om_\tau$,
\end{itemize}
there exists a section $s \in H^0(X_\tau, m L_\tau)$ such that $s(x) = v$ and satisfies
\[
    \int_{X_\tau} |s|_{h^m_\tau}^2 e^{-m u} \om_\tau^n
    \leq C_{\OT} |s(x)|_{h^m_\tau}^2 e^{-m u(x)}.
\]
\end{lem}

\begin{proof}
We first fix some background data.
Let $(B_\af, F_\af, \CB_\af)_\af$ be a locally trivial finite cover of $\CU$ (see Notation~\ref{nota}). 
Take an open set $\CV \Subset \CX \setminus \CZ$ such that $\CB_\af \Subset \CV \Subset \CX \setminus \CZ$ for each $\af$.
Write $U_t := \CU \cap X_t$ and $V_t := \CV \cap X_t$.
Without loss of generality, we can assume that $\CL$ is trivial over each $\CB_\af$ and $h = e^{-g_\af}$ on $\CB_\af$.
Here $g_\af$ is a smooth strictly psh function on $\CB_\af$ and $\om = \ddc g_\af$.

\smallskip
\noindent{\bf Step 1: curvature control.}
For each $\af$, we choose a cutoff $\eta_{\af,0}$ on $B_\af$ such that $\eta_{\af,0}$ is supported in $B_\af$ and $\eta_{\af,0} \equiv 1$ on $\frac{3}{4} B_\af$.
For each $w \in \frac{1}{2} B_\af$, the function 
\[
    \phi_{\af,\tau,w} := (F_{\af,\tau})_\ast \lt(2n \eta_{\af,0} \log |z-w|\rt)
\] 
is quasi-psh on $X_\tau$. 
Up to shrinking $\BD$, there exists a uniform constant $A>0$ such that 
\[
    A \om_t + \ddc_t \phi_{\af,t,w} \geq 0
    \quad\text{and}\quad
    \sup_{X_t} \abs{\db_t \eta_{\af,t}}_{\om_t}^2 \leq A
\]
for all $t \in \BD$, for all $w \in \frac{1}{2} B_\af$, and for all $\af$.

\smallskip
Let $p: \CY \to \CX$ be a resolution of singularities of $\CX$. 
By generic smoothness, up to further shrinking $\BD$, one may assume that the map $p_t := p_{|Y_t}: Y_t \to X_t$ is a resolution of singularities of $X_t$ for all $t \in \BD^\ast$. We can also assume that there exists a constant $C_\CU \geq 1$ such that $C_{\CU}^{-1} \dd V_{\BC^n} \leq F_{\af,t}^\ast \om_t^n \leq C_\CU \dd V_{\BC^n}$ for any $\af$.

\smallskip
Pick a K\"ahler metric $\om_\CY$ on $\CY$ with $\om_\CY \geq p^\ast \om$. 
After rescaling the metric $\om_\CY$, one can assume that for every $t$,
\[
    \om_{\CY,t} + \Ric(\om_{\CY,t}) \geq 0 
    \quad\text{on } p^{-1}(\CV) \cap Y_t.
\]
Since $p^\ast \om$ is semi-positive and big, there is a negative quasi-psh function $\psi \in \PSH(\CY, p^\ast \om)$ such that
\[
    p^\ast \om + \ddc \psi \geq \dt \om_\CY
\]
for some $\dt > 0$ and $\{\psi = -\infty\} = \Exc(p)$.
In addition, one may further assume that 
\[ 
    \psi > -C_\CU 
    \quad\text{and}\quad
    C_{\CU}^{-1} \om_\CY \leq p^\ast \om 
\]
on $p^{-1}(\CV)$. 
For $m \geq m_0 := \lceil (A+1)/\vep \rceil + \lceil 1/\vep\dt\rceil$, we have
\[
    \ii \Ta (m L_\tau, h^m_\tau e^{- (m u + \phi_{\af,\tau,x'})})
    = m (\om_\tau + \ddc_\tau u) + \ddc_\tau \phi_{\af,\tau,x'}
    \geq (m\vep - A) \om_\tau 
    \geq (1 + \lceil 1/\dt \rceil) \om_\tau.
\]
Then we derive 
\begin{align*}
    &\ii \Ta\lt(m p_\tau^\ast L_\tau \otimes K_{Y_\tau}^{-1}, p_\tau^\ast h_\tau^m \otimes h_{\om_{\CY,\tau}} e^{- (m u + \phi_{\af,\tau,x'}) \circ p_\tau} e^{-\lt(\frac{1}{\dt} + 1\rt) \psi_\tau}\rt)\\
    &\geq \lt(\frac{1}{\dt} + 1\rt) (p_\tau^\ast \om_\tau + \ddc_\tau \psi_\tau) + \ii\Ta_{h_{\om_{\CY,\tau}}} (K_{Y_\tau}^{-1})
    \geq (1+\dt) \om_{\CY,\tau} + \Ric(\om_{\CY,\tau}),
\end{align*}
and thus, on $p^{-1}(\CV) \cap Y_\tau$, we obtain the following lower bound for the curvature
\[
    \ii \Ta\lt(m p_\tau^\ast L_\tau \otimes K_{Y_\tau}^{-1}, p_\tau^\ast h_\tau^m \otimes h_{\om_{\CY,\tau}} e^{- (m u + \phi_{\af,\tau,x'}) \circ p_\tau} e^{-\lt(\frac{1}{\dt} + 1\rt) \psi_\tau}\rt)
    \geq \dt \om_{\CY,\tau}.
\]

\smallskip
\noindent{\bf Step 2: extension.}
Fix $x \in \CU \setminus X_0$ and $v \in {mL_\tau}_{|x}$ for $m \geq m_0$.
Note that $x \in \frac{1}{2} B_\af$ for some $\af$. 
We define $x' := F_{\af,\tau}^{-1}(x)$ and denote by $v' \in \BC$ the value of $v$ under the fixed trivialization. 
By the sharp Ohsawa--Takegoshi extension theorem \cite[Thm. 3.1]{Berndtsson_Lempert_2016}, one can find a holomorphic function $f$ on $B_\af$ with $f(x') = v'$ and
\[
    C_\CU^{-1} \int_{B_{\af,\tau}} |(F_{\af,\tau})_\ast f|^2 e^{-m (g_{\af,\tau} + u)} \om_\tau^n
    \leq \int_{B_\af} |f|^2 e^{-m F_{\af, \tau}^\ast (g_{\af,\tau} + u)} \dd V_{\BC^n}
    \leq c_\af |v'|^2 e^{-m g_{\af,\tau}(x)} e^{-m u(x)}, 
\]
where $c_\af$ depends only on the dimension $n$, and $r_\af$ is the radius of $B_\af$.
Therefore, we find a local holomorphic section $\sm \in H^0(B_{\af,\tau}, m L_\tau)$ satisfying $\sm(x) = v$.

\smallskip
Taking $E = m p_\tau^\ast L_\tau \otimes K_{Y_\tau}^{-1}$ and applying Lemma~\ref{lem:Hormander_L2}, we can find a smooth section $\zt$ of $m p^\ast L_\tau$ such that $\db_\tau \zt = \db_\tau p_\tau^\ast \eta_{\af,\tau} \otimes p_\tau^\ast \sm$ and 
\begin{align*}
    \int_{Y_\tau} |\zt|_{p_\tau^\ast h^m_\tau}^2 e^{-m p^\ast u} \om_{\CY,\tau}^n
    &\leq \int_{Y_\tau} |\zt|_{p_\tau^\ast h^m_\tau}^2 e^{- p_\tau^\ast \lt(m u + \phi_{\af,\tau,x'}\rt) - \lt(\frac{1}{\dt} + 1\rt) \psi_\tau} \om_{\CY, \tau}^n\\ 
    &\leq \frac{1}{\dt} \int_{Y_\tau} |\db_\tau p_\tau^\ast \eta_{\af,\tau} \otimes p_\tau^\ast \sm|_{\om_{\CY, \tau} \otimes p_\tau^\ast h^m_\tau}^2 e^{-p_\tau^\ast \lt(m u - \phi_{\af,\tau,x'}\rt) - \lt(\frac{1}{\dt} + 1\rt) \psi_\tau} \om_{\CY, \tau}^n \\
    &\leq \frac{e^{\lt(\frac{1}{\dt} + 1\rt) C_\CU}}{\dt} \int_{p_\tau^{-1} \lt(B_{\af,\tau} \setminus \frac{3}{4} B_{\af,\tau}\rt)} |p_\tau^\ast \lt(\db_\tau \eta_{\af,\tau}\rt)|_{\om_{\CY,\tau}}^2 |p_\tau^\ast \sm|_{p_\tau^\ast h^m_\tau}^2 e^{- p_\tau^\ast \lt(mu + \phi_{\af,\tau,x'}\rt)} \om_{\CY, \tau}^n\\
    &\leq \frac{4 A e^{\lt(\frac{1}{\dt} + 1\rt) C_\CU - 2n} C_\CU^{n+1}}{3 r_\af \dt} \int_{B_{\af,\tau}} |\sm|_{h^m_\tau}^2 e^{-mu} \om_\tau^n \\
    &\leq \frac{4 A e^{\lt(\frac{1}{\dt} + 1\rt) C_\CU - 2n} C_\CU^{n+2} c_\af}{3 r_\af \dt} |\sm(x)|_{h^m_\tau}^2 e^{-m u(x)}.
\end{align*}
From the integrability condition, one has $\zt(x) = 0$ and $s := p_\tau^\ast\lt(\eta_{\af,\tau} \sm\rt) - \zt \in H^0(Y_\tau, m p_\tau^\ast L_\tau)$.
By normality of $X_\tau$, the section $s$ descends to a holomorphic section of $m L_\tau$ on $X_\tau$, which we still denote by $s$.
Moreover,
\begin{align*}
    \int_{X_\tau} |s|^2_{h^m_\tau} e^{-mu} \om_\tau^n 
    &\leq 2 \lt(\int_{X_\tau} |\eta_{\af,\tau} \sm|_{h_\tau^m}^2 e^{-mu} \om_\tau^n + \int_{X_\tau} |s - \eta_{\af,\tau} \sm|_{h_\tau^m}^2 e^{-mu} \om_\tau^n\rt)\\
    &\leq 2 \lt(C_\CU c_\af |s(x)|_{h_\tau^m} e^{-m u(x)} + \int_{Y_\tau} |\zt|^2_{p_\tau^\ast h^m_\tau} e^{-m p_\tau^\ast u} \om_{\CY,\tau}^n\rt)\\
    &\leq 2 C_\CU c_\af \lt(1 + \frac{4 A e^{\lt(\frac{1}{\dt} + 1\rt) C_\CU -2n} C_\CU^{n+1}}{3 r_\af \dt}\rt) |s(x)|_{h^m_\tau}^2 e^{-m u(x)}.
\end{align*}
\end{proof}

Before closing this section, we recall the Monge--Amp\`ere capacity on a K\"ahler space $(X,\om)$. 
Let $E \subset X$ be a Borel subset.
Define
\[
    \CAP_\om (E) := \sup \set{\int_E (\om + \ddc u)^n}{-1 \leq u \leq 0,\, u \in \PSH(X,\om)}.
\]
Then we construct neighborhoods of the singular locus with arbitrarily small capacity in a family $\pi: \CX \to \BD$:

\begin{lem}\label{lem:small_cap_near_sing}
Let $\CZ$ be the singular locus of $\pi$. 
For all $\vep > 0$, there exists an open neighborhood $\CW_\vep$ of $\CZ$ such that for all $t \in \BD_{1/2}$,
\[
    \CAP_{\om_t}(\CW_\vep \cap X_t) < \vep.
\]
\end{lem}

\begin{proof}
Since $\CZ$ is analytic, up to shrinking $\BD$, one can find a $\om$-psh function $\psi \in \PSH(\CX, \om) \cap \CC^\infty(\CX \setminus \CZ)$ such that $\CZ = \{\psi = -\infty\}$. 
In addition, we may assume that there is a constant $C_\psi \geq 0$ such that $- C_\psi \leq \sup_{X_t} \psi_t \leq 0$ for all $t \in \BD_{1/2}$ where $\psi_t = \psi_{|X_t}$.

\smallskip
Fix a constant $M>0$. 
Following the same strategy as in \cite[Prop.~9.10]{GZbook} 
and the sup-$L^1$ comparison, we have the following
\[
    \CAP_{\om_t}(\{\psi_t < -M\}) \leq \frac{1}{M} \lt(\int_{X_t} -\psi_t \om_t^n + n V\rt)
    \leq \frac{V(n + C_{SL} + C_\psi)}{M}.
\]
Therefore, for all $\vep > 0$, the open subset 
\[
    \CW_\vep := \lt\{\psi < - \frac{V(n + C_{SL} + C_\psi)}{\vep}\rt\} \subset \CX
\] 
contains $\CZ$ and for all $t \in \BD_{1/2}$, $\CAP_{\om_t}(\CW_\vep \cap X_t) < \vep$.
\end{proof}

\subsection{Proof of Proposition~\ref{prop:usc_energy_family}}
Without loss of generality, one can assume that $\CL$ is a line bundle and $h$ is a smooth hermitian metric on $\CL$ such that $\om = \ii \Ta_h(\CL)$.
Again, we set $L_k := \CL_{|X_{k}}$, $h_k := h_{|X_k}$, and $\om_k := \om_{|X_k}$.

\smallskip
Before diving into the proof, let us briefly explain the strategy and ideas here.

\smallskip
\noindent{\bf Strategy and ideas:} 
Let us stress that the (mixed) Monge--Amp\`ere operator is not continuous with respect to the $L^1$-convergence (cf. \cite{Lelong_1983}). 
Our proof relies on a non-trivial approximation argument. 
We first observe that for sequences constructed starting from sections of the global line bundle $\CL$, the $L^1$-convergence in families is equivalent to the $\CC^\infty$-convergence. In this case, we deduce the convergence of the Monge--Amp\`ere energies from the local continuity of the (mixed) Monge--Amp\`ere operator with respect to the uniform convergence.

\smallskip
More precisely, given $(u_k^m)_{m \in \BN}$ the algebraic approximations of $u_k$ (Step 0), our strategy proceeds as follows:
\begin{itemize}
    \item Step 1: We establish a uniform upper bound, in $k$ and $m$, for $u_k^m$.
    
    \item Step 2: We show that, on open sets outside the singularities, $u_k^m \geq u_k-C/m$ uniformly for sufficiently large $k$ and $m$.
    
    \item Step 3: We obtain the $\CC^0$-subconvergence of $(u_k^m)_k$ towards a function $u^m$ on the central fiber, and thus, their local Monge--Amp\`ere measure also converges. 
    
    \item Step 4: We prove that any weak limit $\hu$ of $(u^m)_m$ is less than $u_0$. 
    
    \item Step 5: Roughly speaking, we show that 
    {\small
    \[
        "\limsup_{k\to +\infty} \E_k(u_k) \overset{\text{Step 2}}{\leq} \limsup_{m\to +\infty}\limsup_{k\to +\infty}\E_k(u_k^m)+\vep \overset{\text{Step 3}}{=}\limsup_{m\to +\infty}\E_0(u^m)+\vep
        \leq \E_0(\hu) +\vep\overset{\text{Step 4}}{\leq} \E(u_0)+\vep"
    \]
    }%
    for arbitrarily small $\vep>0$.
\end{itemize}

\noindent{\bf Step 0: reductions and algebraic approximations.}
Subtracting a large constant, one can assume that $(u_k)_k$ is a sequence of negative quasi-psh functions by Proposition~\ref{prop:Hartogs_conv_subseq_family}.
It is sufficient to prove the statement for a sequence of uniformly bounded quasi-psh functions $(u_k)_k$.
Indeed, if the sequence $(u_k)_k$ is not uniformly bounded, we take $u_{k,C} := \max\{u_k, -C\}$ for some $C>0$. 
The sequence $u_{k,C}$ converges to $u_{0,C}$ as $k \to +\infty$ in the sense of families. 
Then if the statement holds for a uniformly bounded sequence, we obtain
\[
    \limsup_{k \to + \infty} \E_k(u_k)
    \leq \limsup_{k \to + \infty} \E_k(u_{k,C}) 
    \leq \E_0(u_{0,C}).
\]
Letting $C \to -\infty$, we have the desired upper semi-continuity property. 
On the other hand, one may also assume that $u_k$ is uniformly strictly $\om$-psh; namely, there exists $\dt > 0$ such that $\om_k + \ddc_k u_k \geq \dt \om_k$ for all $k$.
To do so, one can simply consider $(1 - \dt) u_k$, which decreases to $u_k$ as $\dt \to 0^+$ and use the monotonicity of $\E_k$.
From now on, we assume that $-C_{LB} \leq u_k \leq 0$ and $\om_k + \ddc_k u_k \geq \dt \om_k$ for all $k$.

\smallskip
By the upper semi-continuity of $\BD \ni t \mapsto \dim H^0(X_t, m L_t) =: N_{m,t}$ (cf. \cite[Ch.~3, Prop.~1.7]{Banica_Stanasila_1976}), we have $\limsup_{k \to +\infty} N_{m,k} \leq N_{m,0}$. 
We extract a subsequence $(u_{k_j})$ such that 
\[
    \lim_{j \to + \infty} \E_{k_j}(u_{k_j}) = \limsup_{k \to +\infty} \E_k(u_k) 
    \quad\text{and}\quad 
    \lim_{j \to +\infty} N_{m,k_j} = \limsup_{k \to +\infty} N_{m,k}
\] 
for all $m \in \BN$. 
We set $N_m := \lim_{j \to +\infty} N_{m,k_j}$. 
We still denote this subsequence by $(u_k)_k$.

\smallskip
We consider the algebraic approximation of $u_k$ defined as
\[
    u^{m}_k := \frac{1}{m} \log \lt(\sum_{i=1}^{N_{m,k}} |s^m_{i,k}|_{h^m_k}^2\rt)
\]
where $(s^m_{i,k})_i$ is an orthonormal basis of $H^0(X_k, m L_k)$ with respect to the hermitian inner product
\[
    G_{m,k}(s,t)
    := \int_{X_k} \iprod{s}{t}_{h^m_k} e^{-m u_k} \om_k^n.
\]

\smallskip
\noindent{\bf Step 1: a uniform upper bound on the algebraic approximations.}
This part aims to prove that for all $\vep > 0$, there exist $m_0$ and $k_0$ such that 
\begin{equation}\label{eq:alg_app_upper_bdd}
    u^m_k < \vep
\end{equation}
for any $k \geq k_0$ and $m \geq m_0$.
We recall a submean inequality of $\CC^2$ positive psh functions on an analytic set from \cite[p.~194, Cor.~1]{Chirka_Complex_Analytic_Sets}: if $v$ is a positive $\CC^2$ psh function on an $n$-dimensional analytic set $A$ in $\BC^N$, then 
\begin{equation}\label{eq:Chirka_submean}
    v(x) \leq \frac{c_n}{\mult(A,x) r^{2n}} \int_{\BB_r(x) \cap A} v (\ddc |z|^2)^n
\end{equation}
where $\mult(A,x)$ is the multiplicity of $A$ at $x$ and $c_n$ is a universal constant which only depends on $n$.

\smallskip
There exists a finite number of trivializing charts $(U_\af)_\af$ of $\CX$ such that
\begin{itemize}
    \item $\pi^{-1}(\bBD_{1/2}) \Subset \cup_{\af} U_\af$;
    \item $\CL$ is trivial on $U_\af$.
\end{itemize}
Moreover, without loss of generality, we assume that we have an embedding $j_\af: U_\af \hookrightarrow \BB$ where $\BB$ is the unit ball in $\BC$.
Let $g_\af$ be a local potential of $\om$ on $U_\af$. 
Fix $\vep > 0$.
Up to shrinking $U_\af$, one can find a refinement $(V_\af)_\af$ such that 
\begin{itemize}
    \item $V_\af$ still forms a cover of $\pi^{-1}(\bBD_{1/2})$;
    \item $\dist(V_\af, \pl U_\af) = d > 0$ for all $\af$,
\end{itemize}
and one may also assume that
\begin{itemize}
    \item $\osc_{U_\af} g_\af < \vep/2$ for all $\af$;
    \item there exists $C_\om \geq 1$ such that $C_{\om}^{-1} \ddc|z|^2 \leq \om \leq C_\om \ddc|z|^2$ on each $U_\af \hookrightarrow \BB$.
\end{itemize}

\smallskip
Denote by $e^m_\af$ a local trivialization of $m\CL$ on $U_\af$.
We write $s^m_{k,i} = f^m_{k,i,\af} e_\af^m{}_{|X_k}$ where $f^m_{k,i,\af}$ is a holomorphic function on $X_k \cap U_\af$; hence, 
we have $|s^m_{k,i}|^2_{h^m_k} = |f^m_{k,i,\af}|^2 (h^m(e^m_\af, e^m_\af))_{|X_k \cap U_\af}$.
Recall that $\om = -\frac{1}{m} \ddc \log h^m(e^m_\af,e^m_\af)$. 
Thus, one can assume that $h^m(e^m_\af,e^m_\af) = e^{-m g_\af}$. 
Fix $x \in X_k \cap V_\af$. 
By \eqref{eq:Chirka_submean}, we get 
\begin{align*}
    |f^m_{k,i}(x)|^2 
    \leq \frac{c_n}{\mult(X_k, x) d^{2n}} \int_{\BB_d(x) \cap X_k} |f^m_{k,i,\af}|^2 (\ddc|z|^2)^n. 
\end{align*}
Note that for all $x \in X_k$, $\mult(X_k, x)$ is always a positive non-zero integer. 
Recall that by definition, $\int_{X_k} |s^m_{k,i}|_{h^m_k}^2 e^{-m u_k} \om_k^n = 1$. 
We thus have 
\[
    |s^m_{k,i}(x)|_{h_k^m}^2 
    = |f^m_{k,i,\af}(x)|^2 e^{-m g_{\af}(x)} 
    \leq \frac{C_\om^n e^{m \osc_{U_\af} g_\af} c_n}{d^{2n}} \int_{\BB_d(x) \cap X_k} |s^m_{k,i}|_{h_k^m}^2 \om_k^n
    \leq \frac{C_\om^n c_n}{d^{2n}} e^{m \vep/2}.
\]
Then
\[
    u^m_k = \frac{1}{m} \log\lt(\sum_{i=1}^{N_{m,k}} |s^m_{k,i}|_{h^m_k}^2\rt) 
    \leq \frac{\vep}{2} + \frac{\log N_{m,k}}{m} + \frac{n (\log(C_\om/d^2))+\log c_n}{m}. 
\]
Since $\lim_{k\to +\infty}N_{m,k}\leq N_{m,0}$ and $N_{m,0} = \dim H^0(X_0, mL_0) = O(m^n)$, we obtain $u^m_k < \vep$ for all $m, k$ sufficiently large.

\smallskip
\noindent{\bf Step 2: controlling the algebraic approximations from below on the smooth locus.}
Recall that $\om_k + \ddc_k u_k \geq \dt \om_k$ for all $k$.
Fix $\CU \Subset \CX \setminus \CZ$.
From Lemma~\ref{lem:OT_smooth_loc}, we have constants $m_0(\dt, \CU, \CX, \om)$ and $C_{\OT}(\dt, \CU, \CX, \om)$ such that for all $k$ large enough, for all $x \in \CU \cap X_k$, $m \geq m_0$ and $0 \neq v \in m {L_k}_{|x}$, there exists a section $\sm^m \in H^0(X_k, m L_k)$ such that $\sm^m(x) = v$ and
\[
    \int_{X_k} |\sm^m|_{h^m_k}^2 e^{-m u_k} \om_k^n
    \leq C_{\OT} |\sm^m(x)|_{h^m_k}^2 e^{-m u_k(x)}.
\]
Note that the Bergman kernel $\sum_{i=1}^{N_{m,k}} |s^m_{k,i}|_{h_k^m}^2$ at $x \in X_k$ is the peak section, i.e. 
\[
    \sum_{i=1}^{N_{m,k}} |s^m_{k,i}(x)|_{h_k^m}^2
    = \sup \set{|s^m_k(x)|^2_{h_k^m}}{s^m_k \in H^0(X_k,m L_k), \, \int_X |s^m_k|_{h_k^m}^2  e^{-mu_k} \om_k^n = 1}.
\]
Therefore, for $k$ sufficiently large, 
\begin{equation}\label{eq:alg_app_lower_bdd}
    u_k(x) \leq \frac{1}{m} \log\lt(\frac{|\sm^m(x)|_{h^m_k}^2}{\int_{X_k} |\sm^m|_{h^m_k}^2 e^{-m u_k} \om_k^n}\rt) + \frac{\log C_{\OT}}{m}
    \leq u^m_k(x) + \frac{\log C_{\OT}}{m}
\end{equation}
for all $x_k \in \CU \cap X_k$.

\smallskip
\noindent{\bf Step 3: convergence and Monge--Amp\`ere measures.}
In this part, we shall prove that for all $m$ large, the sequence $(u^m_k)_k$ $\CC^0$-subconverges (up to subsequence) to a function $u^m \in \PSH(X_0,\om_0)$ in the family sense, and the unbounded locus of $u^m$ is a subset of $\CZ$. 

\smallskip
Fix $U_0 \Subset X_0^\reg$ and $\CU \Subset \CX \setminus \CZ$ so that $\CU \cap X_0 = U_0$. 
Choose $(B_\af, F_\af, \CB_\af)$ a locally trivial finite cover of $\CU$ so that $\CL$ is trivial on each $\CB_\af$.
Let $e_\af^m$ be a local trivialization of $m \CL$ on $\CB_\af$.
Recall that $\CB_\af \cap X_k$ can be identified with $B_\af$ via $F_{\af,k}: B_\af \to \CB_\af \cap X_k$.
Write $s^m_{k,i} = f^m_{k,i,\af} \cdot {e^m_\af}_{|\CB_\af \cap X_k}$ where $f^m_{k,i,\af} \in \CO(\CB_\af \cap X_k) \simeq \CO(B_\af)$. Observe that for any $m$ fixed $N_{m,k}$ is constant for $k\gg 1$ large. Following the Cauchy estimate as in Step 1, for each $m, \af, i\leq N_m$ fixed, $(|f^m_{k, i, \af}|)_k$ are uniformly bounded in $\CC^0(B_\af)$.
By the Montel theorem, one can extract a subsequence of $(f^m_{k,i,\af})_k$ that converges locally uniformly to a function $f^m_{i,\af} \in \CO(B_\af)$.
This induces a local section $s^m_{i,\af} = f^m_{i,\af} {e^m_{\af}}_{|B_\af} \in H^0(B_\af, m L_0)$.
On another ball $B_{\af'}$ with $B_\af \cap B_{\af'} \neq \emptyset$, up to extracting and relabeling, one obtains another local section $s^m_{i,\af'} \in H^0(B_{\af'}, m L_0)$. 
Since the convergence is locally uniform, $s^m_{i,\af'} = s^m_{i,\af}$ on any compact subset of $B_\af \cap B_{\af'}$ and thus, $s^m_{i,\af'} = s^m_{i,\af}$ on $B_\af \cap B_{\af'}$.
Inductively, we find a section $s^m_i \in H^0(U_0,mL_0)$.
Enlarging $U_0$ towards $X_0^\reg$, we get a section $s^m_i \in H^0(X_0^\reg, mL_0)$. 
By normality of $X_0$, $s^m_i$ extends to a section in $H^0(X_0, mL_0)$ and we still denote it by $s^m_i$.
Therefore, $(\sum_{i=1}^{N_{m,k}} |s^m_{k,i}|_{h^m_k}^2)_k$ $\CC^0$-converges to $\sum_{i=1}^{N_m} |s^m_i|_{h^m_0}^2$ in the family sense. 
From \eqref{eq:alg_app_upper_bdd} and \eqref{eq:alg_app_lower_bdd}, for all $m$ sufficiently large $(u^m_k)_k$ $\CC^0$-converges in the family sense to
\[
    u^m := \frac{1}{m} \log\lt(\sum_{i=1}^{N_m} |s^m_i|_{h^m_0}^2\rt),
\] 
which is not identically $-\infty$.

\smallskip
Again, fix $\CU \Subset \CX \setminus \CZ$ and $(B_\af, F_\af, \CB_\af)$ a locally trivial finite cover of $\CU$.
Recall that the Monge--Amp\`ere operator is continuous along a uniformly convergent sequence of continuous psh functions (cf. \cite[Ch.~III, Cor.~3.6]{Demailly_agbook}).
Hence, for all $j \in\{0, 1, \cdots, n\}$, we have 
\[
    F_{\af,k}^\ast (u^m_k (\ddc (g_{\af, k} + u^m_k))^j \w (\ddc g_{\af, k})^{n-j})
    \rightharpoonup u^m (\ddc (g_{\af, 0} + u^m))^j \w (\ddc g_{\af, 0})^{n-j}
\]
as $k \to +\infty$, where $g_\af$ is a local potential of $\om$ on $\CB_\af$ and $g_{\af, k} = {g_\af}_{|X_k}$.
Let $\chi$ be a smooth function support in $\CB_\af$. 
We obtain 
\[
    \lim_{k \to +\infty} \int_{B_\af} \chi_{|X_0} F^\ast_{\af,k} (u^m_k (\om_k + \ddc u^m_k)^{j} \w \om_k^{n-j})
    = \int_{B_\af} \chi_{|X_0} (u^m (\om_0 + \ddc u^m)^{j} \w \om_0^{n-j}).
\]
By \eqref{eq:alg_app_upper_bdd} and \eqref{eq:alg_app_lower_bdd}, for every $\vep > 0$, we have for all $m, k$ large, we have
\[
    u^m_k < \vep \ \,\text{on}\ \, X_k,
    \quad\text{and}\quad
    -C_{LB} \leq u_k < u^m_k + \vep \ \,\text{on}\ \, \CU \cap X_k.
\]
Therefore, 
{\small
\[
    \abs{\int_{B_\af} (\chi_{|X_k} \circ F_{\af,k} - \chi_{|X_0}) F^\ast_{\af,k} (u^m_k (\om_k + \ddc u^m_k)^{j} \w \om_k^{n-j})}
    \leq \norm{\chi_{|X_k} \circ F_{\af,k} - \chi_{|X_0}}_{L^\infty} (C_{LB} + \vep) V
    \xrightarrow[k \to +\infty]{} 0
\]
}%
and thus,
\begin{equation}\label{eq:MA_conv_smooth_locus}
    \lim_{k \to +\infty} \int_{\CB_\af \cap X_k} \chi_{|X_k} u^m_k (\om_k + \ddc u^m_k)^j \w \om_k^{n-j} 
    = \int_{B_\af} \chi_{|X_0} u^m (\om_0 + \ddc u^m)^j \w \om_0^{n-j}.
\end{equation}

\smallskip
\noindent{\bf Step 4: weak limit of $u^m$.}
For $x \in X_k$, we denote by $s^m_{k,x}$ a peak section of $m L_k$ at $x$ with respect to the $L^2$ inner product $G_{m,k}$. 
Fix $\vep > 0$, an open set $U_0 \Subset X_0^\reg$, and an open set $\CU \Subset \CX \setminus \CZ$ such that $\CU \cap X_0 = U_0$.
Let $(B_\af, F_\af, \CB_\af)_\af$ be a locally trivial finite cover of $\CU$. 
Again, let $g_\af$ be a local potential of $\om$ on $\CB_\af$.
One can assume that $\osc_{\CB_\af} g_\af \leq \vep$ by shrinking and adding balls $B_\af$. 
For any $x \in \frac{1}{2} B_\af$, we set $x_{\af,k} := F_{\af,k}(x)$.
Let $f^m_{k,x_{\af,k}}$ be the holomorphic function representing $s^m_{k,x_{\af,k}}$ under the fixed trivialization of $\CL$ on $\CB_\af$.
Take an arbitrary $r \in (0, r_\af/2)$.
Write $v_{2n}$ for the volume of the unit ball in $\BC^N$. 
Pick a constant $C_\CU > 0$ so that for all $\af$ and $t$ close to $0$, $\dd V_{\BC^n} \leq C_{\CU} F_{\af,t}^\ast \om_t^n$. 
By Cauchy's estimate, 
\begin{align*}
    |F_{\af,k}^\ast f^m_{k,x_{\af,k}} (x)|^2 e^{-m F_{\af,k}^\ast g_{\af,k} (x)}
    &\leq \frac{e^{m \osc_{B_\af} F_{\af,k}^\ast g_{\af,k}}}{v_{2n} r^{2n}} \int_{B(x,r)} |F_{\af,k}^\ast f^m_{k,x_{\af,k}}|^2 e^{-m F_{\af,k}^\ast g_{\af,k}} \dd V_{\BC^n}\\
    &\leq \frac{C_\CU e^{m \vep}}{v_{2n} r^{2n}} 
    \int_{B(x,r)} |F_{\af,k}^\ast f^m_{k,x_{\af,k}}|^2 e^{-m F_{\af,k}^\ast g_{\af,k}} F_{\af,k}^\ast \om_k^n
\end{align*}
and thus,
\begin{align*}
    |s^m_{k,x} (x_{\af,k})|_{h^m_k}^2 
    &= |f^m_{k,x_{\af,k}}(x_{\af,k})|^2 e^{-m g_{\af,k} (x_{\af,k})}\\
    &\leq \frac{C_\CU e^{m \vep}}{v_{2n} r^{2n}} 
    \int_{F_{\af,k}(B(x,r))} |s^m_{k,x_{\af,k}}|_{h^m_k}^2 \om_k^n\\
    &\leq \frac{C_\CU e^{m \lt(\vep + \sup_{F_{\af,k}(B(x,r))} u_k\rt)}}{v_{2n} r^{2n}} 
    \int_{F_{\af,k}(B(x,r))} |s^m_{k,x_{\af,k}}|_{h^m_k}^2 e^{-m u_k} \om_k^n\\
    &\leq  \frac{C_\CU e^{m \lt(\vep + \sup_{F_{\af,k}(B(x,r))} u_k\rt)}}{v_{2n} r^{2n}}.
\end{align*}
We have 
\[
    u^m_k(F_{\af,k}(x)) = \frac{1}{m} \log \lt(|s^m_{k,x_{\af,k}}(x_{\af,k})|_{h^m_k}^2\rt)
    \leq \sup_{F_{\af,k}(B(x,r))} u_k + \vep + \frac{\log (C_\CU/v_{2n}) - 2n \log(r)}{m}
\]
for all $x \in \frac{1}{2} B_{\af}$. 
Since $(F_{\af,k}^\ast u_k)_k$ converges to $u_0$ on $B_\af$, by Hartogs lemma, the above inequality implies that 
\[
    u^m(x) \leq \sup_{B(x,r)} u_0 + \vep + \frac{\log (C_\CU/v_{2n}) - 2n \log(r)}{m}.
\]
for all $x \in \frac{1}{2} B_\af$. 

\smallskip
After extracting a subsequence, we may assume that $(u^m)_m$ converges to a function $\hu \in \PSH(X_0,\om_0)$. 
Taking $m \to +\infty$, we have $\hu(x) \leq \sup_{B(x,r)} u_0 + \vep$.
Letting $r \to 0^+$ and then $\vep \to 0^+$, one derives $\hu(x) \leq u_0(x)$ for all $x \in U_0$.
Then enlarging $U_0$ towards $X_0^\reg$, we have 
\begin{equation}\label{eq:alg_app_weak_limit_upper_bdd}
    \hu \leq u_0    
\end{equation}
on $X_0$.

\smallskip
\noindent{\bf Step 5: upper semi-continuity of the energies.}
Recall that $-C_{LB} \leq u_k \leq 0$ for all $k$.
Set 
\[
    u^m_{k,C} := \max\{u_k^m, -C_{LB}\} 
    = \frac{1}{m} \log\lt(\max\lt\{\sum_{i=1}^{N_{m,k}} |s^m_{k,i}|_{h^m_k}^2 , e^{-mC_{LB}}\rt\}\rt),
\]
\[
    u^m_C := \max\{u^m, -C_{LB}\} 
    = \frac{1}{m} \log\lt(\max\lt\{\sum_{i=1}^{N_m} |s^m_i|_{h^m_0}^2 , e^{-mC_{LB}}\rt\}\rt)
\]
truncations of $u^m_k$ and $u^m$ from below.
One can observe that $(u^m_{k,C})_k$ $\CC^0$-converges to $u^m_C$ in the family sense.

\smallskip
Fix an arbitrary $\vep_1 > 0$. 
From Lemma~\ref{lem:small_cap_near_sing}, we find open subsets $\CW_{\vep_1/2} \Subset \CW_{\vep_1}$ which contain $\CZ$ where $\CW_{\vep_1/2}$ (resp. $\CW_{\vep_1}$) has capacity less than $\vep_1/2$ (resp. $\vep_1$).  
Take 
\[
    \CU_{\vep_1} := \CX \setminus \overline{\CW_{\vep_1/2}}.
\]
We fix $\eta$ a cutoff with support in $\CU_{\vep_1}$ and $\eta \equiv 1 $ on $\CX \setminus \CW_{\vep_1}$.

\smallskip
Then take an arbitrary $\vep_2 > 0$. 
From \eqref{eq:alg_app_lower_bdd} and \eqref{eq:alg_app_upper_bdd}, we have
\[
    u_k < u^m_k + \vep_2 \ \,\text{on}\ \, \CU_{\vep_1} \cap X_k,
    \quad\text{and}\quad
    u^m_k < \vep_2 \ \,\text{on}\ \, X_k
\]
for $m$ and $k$ sufficiently large.
By the monotonicity of the Monge--Amp\`ere energy, we get
\[
    \E_k(u_k)  
    \leq \E_k (\max \{u_k, u^m_{k,C} + \vep_2\} - 2\vep_2) + 2\vep_2.
\]

\smallskip
Let $(B_\af, F_\af, \CB_\af)_\af$ be a locally trivial finite cover of $\CU_{\vep_1}$.
Choose cutoff functions $(\chi_\af)_\af$ such that $\chi_\af$ is supported in $\CB_\af$ and $\sum_\af \chi_\af \equiv 1$ on $\CU_{\vep_1}$, up to shrinking $\BD$. 
Since $\max\{u_k, u^m_{k,C} + \vep_2\} - 2 \vep_2$ is negative on $X_k$ and $u_k < u^m_{k,C} + \vep_2$ on $\CU_{\vep_1} \cap X_k$, we get
{\small
\begin{align*}
    \E_k (\max \{u_{k,\dt}, u^m_{k,C} + \vep_2\} - 2 \vep_2)
    &\leq \frac{1}{(n+1) V} \sum_{j=0}^n \int_{X_k \cap \CU_{\vep_1}} \eta_{|X_k} \lt(u^m_{k,C} - \vep_2 \rt) (\om_k + \ddc u^m_{k,C})^j \w \om_k^{n-j}\\
    &= \frac{1}{(n+1) V} \sum_{j=0}^n \sum_\af \int_{B_{\af,k}} (\eta \chi_{\af})_{|X_k} \lt(u^m_{k,C} - \vep_2 \rt) (\om_k + \ddc u^m_{k,C})^j \w \om_k^{n-j}.
\end{align*}
}%
Hence, 
\[
    \E_k(u_k)
    \leq \frac{1}{(n+1) V} \sum_{j=0}^n \sum_\af \int_{B_{\af,k}} (\eta \chi_{\af})_{|X_k} \lt(u^m_{k,C} - \vep_2 \rt) (\om_k + \ddc u^m_{k,C})^j \w \om_k^{n-j} + 2 \vep_2.
\]
Letting $k \to +\infty$, by \eqref{eq:MA_conv_smooth_locus}, we obtain
\begin{align*}
    \limsup_{k \to + \infty} \E_k(u_k)
    &\leq \frac{1}{(n+1) V} \sum_{j=0}^n \sum_\af \int_{B_{\af}} (\eta \chi_{\af})_{|X_0} \lt(u^m_{C} - \vep_2\rt) (\om_0 + \ddc u^m_{C})^j \w \om_0^{n-j} + 2 \vep_2 \\
    &= \frac{1}{(n+1) V} \sum_{j=0}^n \int_{X_0 \cap \CU_{\vep_1}} \eta_{|X_0} \lt(u^m_{C} - \vep_2 \rt) (\om_0 + \ddc u^m_{C})^j \w \om_0^{n-j} + 2 \vep_2 \\
    & \leq \E_0(u^m_{C}) + \vep_2 + \frac{\lt(C_{LB} + \vep_2\rt)}{(n+1) V} \sum_{j=0}^n \int_{X_0 \cap \CW_{\vep_1}} (\om_0 + \ddc_0 u^m_C)^j \w \om_0^{n-j}.
\end{align*}
Note that we have 
\[
    \frac{1}{2^n} \sum_{j=0}^n \binom{n}{j} \lt(\om_0 + \ddc \frac{v}{B}\rt)^j \w \om_0^{n-j} 
    = \lt(\om_0 + \ddc \frac{v}{2B}\rt)^n
\]
for all $B>1$ and for all $v \in \PSH(X_0,\om_0) \cap L^\infty(X_0)$.
Then one gets
\begin{align*}
    \int_{X_0 \cap \CW_{\vep_1}} (\om_0 + \ddc u_C^m)^j \w \om_0^{n-j}
    &= (C_{LB})^j \int_{X_0 \cap \CW_{\vep_1}} \lt(\frac{\om_0}{C_{LB}} + \ddc \lt[\frac{u_C^m - \vep_2}{C_{LB}}\rt]\rt)^j \w \om_0^{n-j}\\
    &\leq (C_{LB})^n \int_{X_0 \cap \CW_{\vep_1}} \lt(\om_0 + \ddc \lt[\frac{u_C^m - \vep_2}{C_{LB}}\rt]\rt)^j \w \om_0^{n-j}\\
    &\leq (2C_{LB})^n \int_{X_0 \cap \CW_{\vep_1}} \lt(\om_0 + \ddc \lt[\frac{u_C^m - \vep_2}{2C_{LB}}\rt]\rt)^n\\
    &\leq (2C_{LB})^n \CAP_{\om_0} (X_0 \cap \CW_{\vep_1}) 
    < (2C_{LB})^n \vep_1.
\end{align*}
Hence, 
\[
    \limsup_{k \to + \infty} \E_k(u_k)
    \leq \E_0(u^m_C) + \vep_2 + \frac{(2C_{LB})^n (C_{LB} + \vep_2)}{(n+1) V} \vep_1.
\]
By the upper semi-continuity of the Monge--Amp\`ere energy on $X_0$ and \eqref{eq:alg_app_weak_limit_upper_bdd}, we obtain
\begin{align*}
    \limsup_{k \to + \infty} \E_k(u_k)
    &\leq \limsup_{m \to +\infty} \E_0(u^m_C) + \vep_2 + \frac{(2C_{LB})^n (C_{LB} + \vep_2)}{(n+1) V} \vep_1\\
    &\leq \E_0(\max\{\hu, -C_{LB}\}) + \vep_2 + \frac{(2C_{LB})^n (C_{LB} + \vep_2)}{(n+1) V} \vep_1\\
    &\leq \E_0(u_0) + \vep_2 + \frac{(2C_{LB})^n (C_{LB} + \vep_2)}{(n+1) V} \vep_1.
\end{align*}
Finally, letting $\vep_2 \to 0^+$ and $\vep_1 \to 0^+$, we obtain the desired estimate
\[
    \limsup_{k \to +\infty} \E_k(u_k) \leq \E_0(u_0).
\]
This completes the proof of Proposition~\ref{prop:usc_energy_family}.

\section{A Demailly--Koll\'ar result in families}\label{sec:DK_fami}
This section is devoted to the proof of Proposition~\ref{bigprop:DK}.
We first start by discussing Lelong numbers of functions with full Monge--Amp\`ere mass on a singular space. 

\subsection{Lelong number of functions in the finite energy class}
We recall here two different definitions of Lelong numbers on singular complex spaces and how they compare.
We shall prove that functions with full Monge--Amp\`ere mass have zero Lelong numbers in both senses. 

\smallskip
Let $X$ be an $n$-dimensional, normal, complex space.
Fix a point $x \in X$ and local generators $(f_i)_{i \in I}$ of the maximal ideal $\mathfrak{m}_{X,x}$ of $\CO_{X,x}$.
Set $\psi := \sum_{i \in I} |f_i|^2$.
Let $u$ be a psh function defined near $x$.
Recall from \cite[Appx. A]{BBEGZ_2019} that the {\it slope} of $u$ at $x$ is defined by
\[
    s(u,x) := \sup \set{s \geq 0}{u \leq s \log \psi + O(1)}.
\]
In \cite[D\'ef.~3]{Demailly_1982}, Demailly introduced another way of measuring the singularity of $u$ at $x$ by considering
\[
    \nu(u,x) := \lim_{r \to 0} {}^\downarrow \int_{\{\psi < r\}} (\ddc u) \w \lt(\ddc \log \psi \rt)^{n-1}.
\] 
We call it the {\it Demailly--Lelong number} of $u$ at $x$. 
These quantities are both independent of the choice of $(f_i)_i$.

\begin{rmk}
Fix a point $x \in X$.
Take a local embedding $X \xhookrightarrow{\loc.}  \BC^N$ near $x$ and send $x$ to the origin $0 \in \BC^N$. 
Note that $\mathfrak{m}_{X,x} = \mathfrak{m}_{\BC^N,0} / \mathscr{I}_{X,0}$ where $\mathscr{I}_X$ is the ideal sheaf induced by $X$.
From Demailly's comparison theorem \cite[Thm.~4]{Demailly_1982}, one can check that the Demailly--Lelong number can also be expressed as 
\[
    \nu(u,x) = \lim_{r \to 0} {}^\downarrow \int_{\BB_r(0)} \ddc u \w (\ddc \log |z|^2)^{n-1} \w [X], 
\]
where $\BB_r(0)$ is a ball in $\BC^N$ with radius $r$ centered at $0$.
Similarly, one also has
\[
    s(u,x) = \sup \set{s \geq 0}{u_{|X} \leq s \lt(\log |z|^2\rt)_{|X} + O(1)}.
\]
It follows from \cite[Lem.~2.46]{GZbook} that $s(u,x) = \nu(u,x)$ if $x$ is a smooth point. 
However, it is no longer the case in singular settings. 
Precisely, one has the following inequality (cf. \cite[Rmk.~A.5]{BBEGZ_2019}) 
\[
    \nu(u,x) \geq \mult(X,x) s(u,x).
\]
It is conjectured in {\it loc. cit.} that there is a constant $C \geq 1$ independent of $u$ such that 
\begin{equation}\label{eq:conj_comparison_Lelong_num}
    \nu(u,x) \leq C s(u,x).
\end{equation}
\end{rmk}

Recall that $u \in \CE(X,\om)$ is a function in $\PSH(X,\om)$ whose non-pluripolar Monge--Amp\`ere 
\[
    \npp{(\om + \ddc u)^n}
    := {\lim_{j \to+ \infty}}^{\uparrow} 
    \1_{\{u>-j\}} (\om + \ddc \max\{u, -j\})^n 
\]
has full Monge--Amp\`ere mass.
On a singular space, from \cite[Thm~1.1]{BBEGZ_2019}, the slope of $u$ is identically zero. 
We prove that for any function $u \in \CE(X,\om)$, the Demailly--Lelong number is identically zero as well. 

\begin{prop}\label{prop:zero_lelong_num_of_E1_fcn}
Let $(X,\om)$ be an $n$-dimensional, normal, compact K\"ahler space. 
If $u \in \CE(X,\om)$, then $\nu(u,x) = 0$ for every $x \in X$.
\end{prop}

\begin{proof}
Fix a point $x \in X$.
Note that $\nu(\rho,x) = 0$ for any smooth psh function $\rho$ defined near $x$. 
Let $\CI$ be the ideal sheaf of the point $x$. 
By Hironaka's theorem, one can find a log-resolution 
\[
    p: (\wX, \CO_{\wX}(-D)) \to (X,\CI)
\] 
so that $\wX$ is smooth, $p^{-1} \CI = \CO_{\wX}(-D)$, and $D = \sum_l a_l D_l$ is effective where $(D_l)_l$ are prime divisors. 
Let $\rho$ be a smooth potential of $\om$ near $x$.
The function $v := \rho + u$ is psh near $x$.
By Forn{\ae}ss--Narasimhan~\cite[Thm~5.5]{FN_1980}, there is a sequence of smooth psh functions $(v_j)_j$ decreasing towards $v$.
Recall that $\psi := \sum_{i \in I} |f_i|^2$ where $(f_i)_{i \in I}$ are local generators of the maximal ideal $\mathfrak{m}_{X,x}$ of $\CO_{X,x}$. 
Fix $r>0$ and a cutoff function $\chi$ which is identically $1$ on $\{\psi < r\}$ and has support in $\{\psi < 2r\}$.
Then we consider the following approximation of $\nu(u,x)$:
\[
    L_{j, \vep_1, \cdots, \vep_{n-1}} 
    = \int_{\{\psi<2r\}} \chi \ddc v_j  \w \ddc \log (\psi + \vep_1) \w \cdots \w \ddc \log (\psi + \vep_{n-1}),
\]
for $(\vep_k)_{1 \leq k \leq n-1} \in (0,1]^{n-1}$. 
By the continuity of the complex Monge--Amp\`ere operator along monotone sequences, we have 
\[
    \nu(u,x) \leq \int_{\{\psi<2r\}} \chi \ddc v  \w (\ddc \log \psi)^{n-1}
    = \lim_{j \to +\infty} \lim_{\vep_1 \to 0} \cdots \lim_{\vep_{n-1} \to 0} L_{j, \vep_1, \cdots, \vep_{n-1}}.
\]

\smallskip
Pulling back all data to $\wX$, we obtain  
\[
    L_{j, \vep_1, \cdots, \vep_{n-1}} = \int_{p^{-1}(\{\psi < 2r\})} (\chi \circ p) p^\ast \ddc v_j \w p^\ast \ddc \log(\psi + \vep_1) \w \cdots \w p^\ast \ddc \log (\psi + \vep_{n-1}).
\]
Note that
\[
    p^\ast \ddc \log \psi = \sum_l a_l [D_l] + \ta
\]
where $\ta$ is a smooth semi-positive $(1,1)$-form defined near $p^{-1}(x)$.
Choose a K\"ahler metric $\wom$ on $\wX$ so that $p^\ast \om \leq \wom$ on $\wX$ and $\ta \leq \wom$ near $p^{-1}(x)$.
Taking $\vep_{n-1} \to 0$, one can derive that $L_{j, \vep_1, \cdots, \vep_{n-1}}$ converges to 
\begin{align*}
    L_{j, \vep_1, \cdots, \vep_{n-2}} 
    &= \int_{p^{-1}(\{\psi < 2r\})} (\chi \circ p) p^\ast \ddc v_j \w \bigwedge_{k=1}^{n-2} p^\ast \ddc \log (\psi + \vep_k) \w \lt(\sum_l a_l [D_l] + \ta\rt)\\
    &= \int_{p^{-1}(\{\psi < 2r\})} (\chi \circ p) p^\ast \ddc v_j \w \bigwedge_{k=1}^{n-2} p^\ast \ddc \log (\psi + \vep_k) \w \ta;
\end{align*}
here the second equality comes from the constancy of $p^\ast v_j$ along each $D_l$.
Inductively, one can deduce
\begin{align*}
    \lim_{j \to +\infty} \lim_{\vep_1 \to 0} \cdots \lim_{\vep_{n-1} \to 0} 
    L_{j, \vep_1, \cdots, \vep_{n-1}} 
    &= \int_{p^{-1}(\{\psi < 2r\})} (\chi \circ p) p^\ast (\om + \ddc u) \w \ta^n\\
    &\leq \int_{p^{-1}(\{\psi < 2r\})}  (\wom + \ddc p^\ast u) \w \wom^{n-1}.
\end{align*}
Letting $r \to 0$, we obtain
\[
    \nu(u,x) \leq \int_{p^{-1}(x)} (\wom + \ddc p^\ast u) \w \wom^{n-1}.
\]

\smallskip
One can easily check that $p^\ast u \in \CE(\wX, p^\ast \om)$. 
By \cite[Thm~1.1 (ii)]{DDL_2018}, we have 
\[
    \CE(\wX, p^\ast \om) \cap \PSH(\wX, \wom) \subset \CE(\wX, \wom),
\] 
and thus $p^\ast u \in \CE(\wX, \wom)$. 
Note that $(\wom + \ddc p^\ast u) \w \wom^{n-1}$ is a mixed Monge--Amp\`ere measure of $p^\ast u \in \CE(\wX, \wom)$ and $0 \in \CE(\wX, \wom)$, so it does not charge pluripolar sets (see e.g. \cite[Cor.~2.15]{BEGZ_2010}).
The set $p^{-1}(x)$ is analytic, so is pluripolar; hence 
\[
    \int_{p^{-1}(x)} (\wom + \ddc p^\ast u) \w \wom^{n-1} = 0
\] 
and $\nu(u,x) = 0$.
\end{proof}

\subsection{A Skoda--Zeriahi type estimate}
In this section, we establish the following Skoda--Zeriahi type estimate along a convergent sequence of $\CE$-functions in families:

\begin{lem}\label{lem:Skoda_for_E1_seq}
In Setting~\ref{sett:general_sett_normal_fibre}, if $(u_k)_k \subset \CE_{\fibre}(\CX, \om)$ converges to $u_0 \in \CE(X_0, \om_0)$, then for all $\gm > 0$, there is a constant $A_\gm > 0$ such that
\[
    \int_{X_k} e^{-\gm u_k} \om_k^n \leq A_\gm.
\]
\end{lem}

\begin{proof}
The proof follows the same strategy as in \cite[Thm.~3.1]{Zeriahi_2001} and \cite[Thm.~4.1]{Pan_2023}.
An extra difficulty is that we need to establish arbitrarily small upper bounds on local projective masses along such a sequence. 

\smallskip
\noindent{\bf Step 0: reductions.}
Without loss of generality, we may assume that $(u_k)_k$ and $u_0$ are negative. 
Also, we assume that $\BB_{1/2} \subset \BB \subset \BC^N$ and the collection of $\BB_{1/2}$ cover $\CX$ up to shrinking $\BD$.
On each $\BB$, one can write $\om = \ddc g$ for some smooth strictly psh function $g$ defined in a neighborhood of $\overline{\BB}$. 
Since the collection of $\BB$ is finite, one can assume that $-C \leq g \leq 0$ on each $\BB$ for a uniform constant $C>0$.
By \cite[Thm.~5.5]{FN_1980}, the function $g + u_k$ is a decreasing limit of negative smooth psh functions $(v_{t_k, l})_{l \in \BN}$ on $\BB \cap X_k$.
By the monotone convergence theorem, for any $\vep > 0$ and each $k$ fixed, one can find $l_k$ such that 
\[
    \int_{\BB_{1/2} \cap X_k} e^{-\gm (g+u_k)} \om_k^n 
    \leq \int_{\BB_{1/2} \cap X_k} e^{-\gm (v_{t_k, l})} \om_k^n + \vep
    \quad\text{and}\quad
    \norm{(g+u_k) - v_{t_k,l}}_{L^1(\BB \cap X_k)} < \frac{\vep}
    {2^{k+1}} 
\]
for all $l \geq l_k$.
We may assume that $l_{k+1} > l_{k}$ for all $k \in \BN$.
Define $v_k := v_{t_k, l_k}$ which is a smooth psh function on $\BB \cap X_k$.
To prove Lemma~\ref{lem:Skoda_for_E1_seq}, it suffices to show that for each $\gm > 0$ and on each $\BB_{1/2}$, there exists a constant $A_\gm'$ such that for all $k \in \BN$, the following estimate holds 
\begin{equation*}
    \int_{\BB_{1/2} \cap X_k} e^{-\gm v_{k}} \om_k^n \leq A'_\gm.
\end{equation*}

\smallskip
\noindent{\bf Step 1: controlling Demailly--Lelong numbers.}
Let $\Phi_x$ denote the automorphism of $\BB$ sending $x$ to the origin and consider $G_x(z) := \log |\Phi_x(z)|$ the pluricomplex Green function of $\BB$.
Define 
\[
    \nu(v_k, x, r) := \int_{\overline{D(x,r)}} \ddc v_k \w (\ddc G_x)^{n-1} \w [X_k]
\]
where $D(x,r) := \set{\zt \in \BB}{|\Phi_x(\zt)| < r}$.
Arguing as in \cite[p.~65]{GZbook}, one can check that 
\[
    \nu(v_k, x, r) = \frac{1}{r^{2n-2}} \int_{\overline{D(x,r)}} \ddc v_k \w (\ddc |\Phi_x|^2)^{n-1} \w [X_k].
\]
We shall show an upper semi-continuity property: 
\begin{equation}\label{eq:Lelong_number_usc}
    \limsup_{k \to + \infty} \nu(v_{k}, x_{k}, r) 
    \leq \nu(g+u_0, x_0, r)
\end{equation}
for any sequence $x_{k} \in \overline{\BB}_{1/2}$ converging to $x_0 \in \overline{\BB}_{1/2}$.
Arguing as in Proposition~\ref{prop:hartogs_currents}, one shows that
\begin{equation}\label{eq:weak_conv_by_smooth}
    \ddc v_k \w [X_k] \rightharpoonup \ddc (g+u_0) \w [X_0]
\end{equation}
on $\BB$.
Fix $\vep > 0$.
For all $k$ large enough, we have $D(x_k, r)\subset D(x_0, r + \vep)$ and
\[
    (1-a_k) \ddc |\Phi_{x_0}|^2
    \leq \ddc |\Phi_{X_k}|^2 
    \leq (1+a_k) \ddc |\Phi_{x_0}|^2
\]
for some $a_k \searrow 0$.
Therefore, one has
\begin{align*}
    \nu(v_{k}, x_k, r) &= \frac{1}{r^{2n-2}} \int_{\overline{D(x_k,r)}} \ddc v_{k} \w (\ddc |\Phi_{x_k}|^2)^{n-1} \w [X_k]\\
    &\leq \frac{(1+a_k)^n}{r^{2n-2}} \int_{\overline{D(x_0,r+\vep)}} \ddc v_{k} \w (\ddc |\Phi_{x_0}|^2)^{n-1} \w [X_k].
\end{align*}
By \eqref{eq:weak_conv_by_smooth}, we have
\[
    \limsup_{k \to + \infty} \int_{\overline{D(x_0,r+\vep)}} \ddc v_{k} \w (\ddc |\Phi_{x_0}|^2)^{n-1} \w [X_k]
    \leq \int_{\overline{D(x_0,r+\vep)}} \ddc (g+u_{0}) \w (\ddc |\Phi_{x_0}|^2)^{n-1} \w [X_{0}].
\]
This implies that
\[
    \limsup_{k \to +\infty} \nu(v_{k}, x_k, r) 
    \leq \lt(\frac{r+\vep}{r}\rt)^{2n-2} \nu(g+u_0, x_0, r+\vep) 
    \xrightarrow[\vep \to 0^+]{} \nu(g+u_0, x_0, r).
\]

\smallskip
Note that $\nu(g+u_0, x, r)$ decreases to $\nu(g+u_0,x)$ as $r \to 0^+$ which is the Lelong number of $u_0$ at $x$.
From Proposition~\ref{prop:zero_lelong_num_of_E1_fcn}, we have $\nu(g+u_0, x) = 0$ for any $x \in \BB_{1/2}$.
By \eqref{eq:Lelong_number_usc} and Dini's lemma, $\nu(g+u_0, x, r)$ converges uniformly (independent of $x \in \overline{\BB}_{1/2}$) to $0$ as $r \to 0^+$. 
Therefore, for all $\nu < 2/\gm$, for all $x \in \BB_{1/2}$, there is a uniform $r_0 > 0$ such that $\nu(g+u_0, x, r_0) < \nu/2$ for any $x \in \overline{\BB}_{1/2}$.
By \eqref{eq:Lelong_number_usc}, one can find $M \in \BN$ such that for all $x \in \overline{\BB}_{1/2}$ and for all $k > M$, 
\[
    \nu(v_{k}, x, r_0) \leq \frac{3 \nu}{4}.
\]
Adding $\vep (|z|^2 - 1)$ for some small $\vep > 0$, one can find a uniform constant $c_\vep > 0$ such that 
\[
    c_\vep \leq \nu(v_{k} + \vep (|z|^2-1), x, r_0) \leq \nu
\]
for all $x \in \BB_{1/2} \cap X_k$.

\smallskip
\noindent{\bf Step 2: conclusion.}
We have $(\ddc G_x)^n \w [X_t] \geq \dt_x$ (cf. \cite[p.~2719]{DGG2020}) and thus 
\begin{equation}\label{eq:Poisson_Szego_inequality}
\begin{split}
    \psi_t(x) &\geq \int_{\BB} \psi_t (\ddc G_x)^n \w [X_t]\\
    &= \underbrace{\int_{\BB} G_x \ddc \psi_t \w (\ddc G_x)^{n-1} \w [X_t]}_{=: I_{\psi_t}(x)} + \underbrace{\int_{\pl\BB} \psi_t \dc G_x \w (\ddc G_x)^{n-1} \w [X_t]}_{=: J_{\psi_t}(x)}
\end{split}
\end{equation}
for $0 \geq \psi_t \in \PSH(\BB \cap X_t) \cap \CC^\infty(\BB \cap X_t)$.
Define $D(x,r) := \set{z \in \BB}{|\Phi_x(z)| \leq r}$.
We separate $I_t$ as 
\begin{equation}\label{eq:split_Lelong_number}
    I_t(x) = \underbrace{\int_{\overline{D(x,r)}} G_x \ddc \psi_t \w (\ddc G_x)^{n-1} \w [X_t]}_{=: K_{\psi_t}(x)} + \underbrace{\int_{\BB \setminus \overline{D(x,r)}} G_x \ddc \psi_t \w (\ddc G_x)^{n-1} \w [X_t]}_{=: L_{\psi_t}(x)}.
\end{equation}

\smallskip
Write $w_{k} = v_{k} + \vep (|z|^2 -1)$.
As in \cite[middle of p.~2720]{DGG2020}, the second term $|J_t|$ in \eqref{eq:Poisson_Szego_inequality} of $w_k$ is dominated by a uniform constant $C>0$.
The second term $|L_t|$ in \eqref{eq:split_Lelong_number} of $w_k$ is uniformly bounded by $C (\log r_0)/r_0^{2n-2}$ for some constant $C>0$ (cf. \cite[p.~70]{GZbook}).
Following the arguments in \cite[pp. 19-20]{Pan_2023} allows one to complete the proof.
\end{proof}

\subsection{A Demailly--Koll\'ar type result for functions with full Monge--Amp\`ere mass}
\subsubsection{Adapted measures}\label{sec:adapted_measure}
In this part, we recall the notion of adapted measures as in \cite[Sec.~6]{EGZ_2009} and a uniform $L^p$ estimate of the canonical densities as in \cite[Lem.~4.4]{DGG2020} under the following setting:

\begin{taggedbigset}{(klt)}\label{sett:klt}
Under Setting~\ref{sett:general_sett_normal_fibre}, further assume that $\CX$ is $\BQ$-Gorenstein and $X_0$ has at most klt singularities.
\end{taggedbigset}

\begin{rmk}\label{rmk:klt_open}
In Setting~\ref{sett:klt}, by inversion of adjunction (cf. \cite[Thm.~4.9]{Kollar_2013}), $\CX$ has klt singularities near $X_0$.
Moreover, $X_t$ has klt singularities for all $t$ close to $0$ (cf. \cite[Cor.~4.10]{Kollar_2013}).    
\end{rmk}

Let $m \geq 1$ be an integer such that $m K_{\CX/\BD}$ is Cartier.
Fix $h$ a smooth hermitian metric on $m K_{\CX/\BD}$. 
This induces a curvature form $\ii \Ta(m K_{\CX/\BD}, h)$ and it is locally written as $-\ddc \log |\Om|_h^{2/m}$ where $\Om$ is a local trivialization of $m K_{\CX/\BD}$.
As in \cite[Sec.~5]{EGZ_2009}, the curvature form is canonically attached to an element $H^0(\CX, \CC^\infty_\CX/\PH_\CX)$ where $\CC^\infty_\CX$ (resp. $\PH_\CX$) is the subsheaf of continuous functions on $\CX$ that are local restrictions of smooth functions (resp. pluriharmonic functions) under local embeddings. 
The first Chern class of $K_{\CX/\BD}$, $c_1(K_{\CX/\BD})$, is the image of $\frac{\ii}{m} \Ta(m K_{\CX/\BD}, h)$ in $H^1(\CX,\PH_\CX)$ via the map $[\cdot]$
\[
    H^0(\CX, \CC^\infty_\CX) \to H^0(\CX, \CC^\infty_\CX/ \PH_\CX) \xrightarrow[]{[\cdot]} H^1(\CX, \PH_\CX) \to 0
\]
induced by the short exact sequence $0 \to \PH_\CX \to \CC^\infty_\CX \to \CC^\infty_\CX/ \PH_\CX \to 0$. 
Note that every representative in $c_1(K_{\CX/\BD})$ is a curvature form $\frac{1}{m}\Ta(mK_{\CX/\BD},h)$ of a smooth hermitian metric $h$ on $K_{\CX/\BD}$ (cf. \cite[Prop.~5.12]{EGZ_2009}). 
A similar definition and properties hold also for $-K_{\CX/\BD}$.

\smallskip
Given $\Om$ a local trivialization of $m K_{\CX/\BD}$, we consider the following adapted measure on $X_t^\reg$
\[
    \mu_t := \ii^{n^2} \lt(\frac{\Om_t \w \overline{\Om_t}}{|\Om_t|_{h_t}^2}\rt)^{\frac{1}{m}} 
\]
where $\Om_t$ (resp. $h_t$) is the restriction of $\Om$ (resp. $h$) to $X_t$.
Note that $\mu_t$ does not depend on the choice of $\Om$.
By a klt version of \cite[Lem.~4.4]{DGG2020}, $\mu_t = f_t \om_t^n$ with $0 \leq f_t \in L^p(X_t, \om_t^n)$ for some $p > 1$ and there exists a constant $C_p > 0$ such that for all $t \in \bBD_{1/2}$,
\begin{equation}\label{eq:klt_Lp_estimate}
    \int_{X_t} f_t^p \om_t^n \leq C_p.
\end{equation}

\subsubsection{Demailly--Koll\'ar type result}
We first recall the theorem of Demailly and Koll\'ar:

\begin{thm}[{\cite[Main~Thm.~0.2~(2)]{Demailly_Kollar_2001}}]\label{thm:Demailly_Kollar}
Let $X$ be a complex manifold and $K$ a compact subset in $X$.
If $(u_j)_j$ is a sequence of quasi-psh functions which converges to a quasi-psh function $u$ in $L^1_\loc$, then for all $c < c_K(u)$, over some neighborhood $U$ of $K$ 
\[
    e^{-2 c u_j} \xrightarrow{L^1(U)} e^{-2 c u}
\]
where $c_K(u) := \sup \set{c \geq 0}{e^{-2cu}\text{ is $L^1$ on a neighborhood of $K$}}$. 
\end{thm}

We establish the following continuity result in the spirit of Theorem \ref{thm:Demailly_Kollar}.
\begin{prop}\label{prop:Demailly_Kollar_family}
In Setting~\ref{sett:klt}, if $(u_k)_k \subset \CE_{\fibre}(\CX, \om)$ converges to $u_0 \in \CE(X_0, \om_0)$, then for any $\gm > 0$,
\[
    \int_{X_k} e^{-\gm u_k} \dd \mu_k
    \xrightarrow[k \to +\infty]{} \int_{X_0} e^{-\gm u_0} \dd \mu_0. 
\]
\end{prop}

\begin{proof}
Without loss of generality, we may assume that $(u_k)_k$ and $u_0$ are negative.
Let $D \geq 0$ be a constant such that $\sup_{X_k} u_k \geq -D$ for all $k \in \BN$ and $\sup_{X_{0}} u_{0} \geq -D$.
For $\vep>0$, we find an open neighborhood $\CW_\vep$ of $X_0^\sing$ such that 
\begin{equation}\label{eq:vol_small_open_near_sing_2}
    \Vol_{\om_{t}}(X_t \cap \CW_\vep) < \vep
\end{equation}
for all $t$ sufficiently close to $0$.
Also, there is a smaller open set $\CW_{\vep/2}$ with $X_0^\sing \subset \CW_{\vep/2} \Subset \CW_\vep$. 
Let $(B_\af, F_\af, \CB_\af)_\af$ be a locally trivial finite cover of $\CX \setminus \overline{\CZ_\vep}$.
Up to shrinking $\BD$, one can find $(\chi_\af)_{\af \in J} \cup \{\chi_\CW\}$ cutoffs satisfying the following conditions:
\begin{itemize}
    \item $\chi_\CW$ is supported in $\CW_\vep$ and for each $\af \in J$, $\chi_\af$ is supported in $\CB_\af$;
    \item $\chi_\CW + \sum_{\af \in J} \chi_j \equiv 1$ on $\pi^{-1}(\bBD_r)$ for some $r>0$ sufficiently small.
\end{itemize} 
Define $\chi_{\af,t} := {\chi_\af}_{|X_t}$ and $\chi_{\CW,t} = {\chi_\CW}_{|X_t}$.
We estimate the desired quantity by two terms:
{\small
\begin{align*}
    &\abs{\int_{X_k} e^{-\gm u_k} \dd \mu_k - \int_{X_0} e^{-\gm u_0} \dd \mu_0}\\ 
    &\leq \underbrace{\abs{\sum_{\af} \lt(\int_{B_{\af, k}} \chi_{\af,k} e^{-\gm u_k} f_k \om_k^n - \int_{B_\af} \chi_{\af,0} e^{-\gm u_0} f_0 \om_0^n\rt)}}_{=: \RN{1}} 
    + \underbrace{\abs{\int_{X_k \cap \CW_\vep} e^{-\gm u_k} f_k \om_k^n} + \abs{\int_{X_0 \cap \CW_\vep} e^{-\gm u_0} f_0 \om_0^n}}_{=: \RN{2}}.
\end{align*}}
For term $\RN{1}$, we have
{\footnotesize
\begin{align*}
    \RN{1} 
    &\leq \sum_{\af} \abs{\int_{B_\af} \lt(e^{-\gm u_k \circ F_{\af,k}} \frac{F^\ast_{\af, k} (\chi_{\af,k} f_k \om_k^n)}{\om_0^n} - e^{-\gm u_0} \chi_{\af,0} f_0\rt) \om_0^n}\\
    &\leq \underbrace{\sum_{\af} \abs{\int_{B_\af} \lt(e^{-\gm u_k \circ F_{\af,k}} - e^{-\gm u_0}\rt) \frac{F^\ast_{\af,k} (\chi_{\af,k} f_k \om_k^n)}{\om_0^n} \om_0^n}}_{=: \RN{3}} 
    + \underbrace{\sum_{\af} \abs{\int_{B_\af} e^{-\gm u_0} \lt(\frac{F^\ast_{\af,k} (\chi_{\af,k} f_k \om_k^n)}{\om_0^n}  - \chi_{\af,0} f_0\rt)\om_0^n}}_{=: \RN{4}}.
\end{align*}}
Let $\Om_\af$ be a local trivialization of $-K_{\CX/\BD}$ on $\CB_\af$.
One can write 
\[
    \frac{F^\ast_{\af,t} (f_t \om_t^n)}{\om_0^n}
    = F^\ast_{\af,t} \lt(\ii^{n^2} \res{\frac{\Om_\af \w \overline{\Om_\af}}{|\Om_\af|^2_h}}{X_{t} \cap \CB_\af}\rt) \cdot \frac{1}{\om_0^n}. 
\]
Then on $B_\af$, the sequence $\lt(\frac{F^\ast_{\af,k} (\chi_{\af,k} f_k \om_k^n)}{\om_0^n}\rt)_k$ converges smoothly to $\chi_{\af,0} f_0$. 
Hence, for each $\af$, there is a constant $C_\af > 0$ such that 
\[
    \frac{F^\ast_{\af,k} ( \chi_{\af,k} f_k \om_k^n)}{\om_0^n} \leq C_\af,
\]
and one also has 
\[
    \sup_{B_\af} \abs{\frac{F^\ast_{\af,k} (\chi_{\af,k} f_k \om_k^n)}{\om_0^n} - \chi_{\af,0} f_0} 
    \leq \frac{\vep}{|J|}
\]
for all $k$ sufficiently large.
By Theorem~\ref{thm:Demailly_Kollar}, for all $\gm > 0$,  $e^{-\gm u_k \circ F_{\af,k}}$ converges to $e^{-\gm u_0}$ on $L^1(B_\af)$.
Namely, for all $k$ sufficiently large, we have  
\[
    \int_{B_\af} \abs{e^{-\gm u_k \circ F_{\af,k}} - e^{-\gm u_0}} \om_0^n \leq \frac{\vep}{C_\af |J|}. 
\]
These yield the following estimates
\[
    \RN{3} \leq \sum_{j \in J} C_\af \int_{B_\af} \abs{e^{-\gm u_k \circ F_{\af,k}} - e^{-\gm u_0}} \om_0^n 
    \leq \vep,
\]
and
\[
    \RN{4} \leq \sum_{\af \in J} \frac{\vep}{|J|} \int_{X_0} e^{-\gm u_0} \om_0
    \leq \int_{X_0} e^{-\gm u_0} \om_0^n \vep.
\]

\smallskip
Next, fix two constants $q, r> 1$ such that $\frac{1}{p} + \frac{1}{q} + \frac{1}{r} = 1$ where $p$ is given by \eqref{eq:klt_Lp_estimate}.
By H\"older's inequality, Lemma~\ref{lem:Skoda_for_E1_seq} and \eqref{eq:vol_small_open_near_sing_2}, we obtain
{\footnotesize
\begin{align*}
    \RN{2} 
    &\leq C_p^{1/p} \lt[\lt(\int_{X_k \cap \CW_\vep} e^{-\gm q u_k} \om_k^n \rt)^{1/q} \Vol_{\om_k}^{1/r}(X_k \cap \CW_\vep) + \lt(\int_{X_{0} \cap \CW_\vep} e^{-\gm q u_{0}} \om_{0}^n \rt)^{1/q} \Vol_{\om_{0}}^{1/r} (X_0 \cap \CW_\vep)\rt] \\
    &\leq 2 C_p^{1/p} A_{\gm q}^{1/q} \vep^{1/r}.
\end{align*}}
All in all, one gets 
\[
    \abs{\int_{X_k} e^{-\gm u_k} \dd \mu_k - \int_{X_0} e^{-\gm u_0} \dd \mu_0} 
    \leq \RN{2} + \RN{3} + \RN{4}
    \leq 2 C_p^{1/p} A_{\gm q}^{1/q} \vep^{1/r} + \lt(1 + \int_{X_0} e^{-\gm u_0} \om_0^n\rt) \vep 
\]
which completes the proof.
\end{proof}

\section{Geometric applications}\label{sec:applications}
We are now ready to prove our main results, Theorems~\ref{bigthm:openness_KE}~and~\ref{bigthm:L_infty_est} from the introduction.

\subsection{Openness of K\"ahler--Einstein metrics}
In this section, we prove Theorem \ref{bigthm:openness_KE}.

\begin{taggedbigset}{(KE)}\label{set:K_stab_fami}
Under Setting~\ref{sett:klt}, further, assume that $-K_{\CX/\BD}$ is relatively ample, $\Aut(X_0)$ is discrete, and $X_0$ admits a K\"ahler--Einstein metric .
\end{taggedbigset}

\begin{thm}\label{thm:openness_KE}
Under Setting \ref{set:K_stab_fami}, for all $t$ sufficiently close to $0$, $X_t$ admits a K\"ahler--Einstein metric.
\end{thm}

By the singular version of the Yau--Tian--Donaldson correspondence \cite{Li_Tian_Wang_2022, Liu_Xu_Zhuang_2022}, Theorem \ref{bigthm:openness_KE} follows from the openness of K-stability \cite[Thm.~A]{Blum_Liu_2022}.

\smallskip
We prove that $X_t$ admits a unique singular K\"ahler--Einstein metric for any $t$ sufficiently close to $0$, without relying on the Yau--Tian--Donaldson correspondence and the openness of K-stability. 

\begin{proof}[Proof of Theorem~\ref{thm:openness_KE}]
Proposition~\ref{bigprop:usc_energy} and Proposition \ref{bigprop:DK} lead to the lower semi-continuity of the family of functionals $\{\D_t\}_{t}$ over normalized potentials with uniformly bounded Monge--Ampère energy. 
The strategy for proving Theorem \ref{thm:openness_KE} consists of  exploiting the coercivity of the Ding functional $\D_0$ on the central fibre to deduce that $\D_t$ necessarily admits a minimizer if $t$ is sufficiently close to $0$.

\smallskip
Recall that from Theorem \ref{Summ_Thm} there exist constants $A\in (0,1), B\geq 0$ such that 
\begin{equation}\label{eq:pf_thmA_coercivity}
    \D_0(u) \geq A (-\E_0(u)) - B
\end{equation}
for all 
\[
    u \in \CE^1_{\nmlz}(X_0, \om_0) := \set{v \in \CE^1(X_0,\om_0)}{\sup_{X_0} v =0}.
\] 
Here the Ding functionals $\D_t: \CE^1(X_t,\om_t)\to \BR$ are defined as
\[
    \D_t(u)=-\E_t(u)-\log\int_{X_t}e^{-u}c_t\dd \mu_t
\]
where $c_t:=1/\int_{X_t} \dd \mu_t$ makes $c_t\dd \mu_t$ a probability measure (see Section \ref{sec:adapted_measure}). 
One can check that $c_t$ is continuous in $t \in \BD$ (cf. Proposition~\ref{bigprop:DK}), and bounded away from $0$ and $+\infty$.
Let also $C_{SL}>0$ be the uniform constant given by the sup-$L^1$ comparison in Theorem \ref{thm:SL_and_Skoda_in_family}. 
We claim the following:

\begin{claim}\label{Claim}
For any $\vep>0$ there exists $r_\vep>0$ such that, for all $t\in \BD_{r_\vep}$, if $\D_t(u)\leq 0$ for $u\in \CE^1_{\nmlz}(X_t,\om_t)$ then $-\E_t(u)< \frac{B+C_{SL}}{A}+\vep$.
\end{claim}

\begin{proof}[Proof of Claim~\ref{Claim}]
Fix $\vep>0$, set $M:=\frac{B+C_{SL}}{A}$ and assume by contradiction that there exists a sequence $u_{t_k}\in \CE^1_{\nmlz}(X_{t_k},\om_{t_k})$ such that $-\E_{t_k}(u_{t_k})\geq M+\vep$. 
To lighten the notation we set $X_k:=X_{t_k}, \om_k:=\om_{t_k}$ and similarly for other quantities.
Let $g_k(s)$ be the unit-speed geodesic connecting $0$ and $u_k$ in $\CE^1_{\nmlz}(X_k,\om_k)$ and set $v_k:= g_k(M+\vep)$.
By Proposition~\ref{prop:Hartogs_conv_subseq_family}, up to extracting a subsequence, $v_k$ converges to a function $v \in \PSH(X_0,\om_0)$ in the family sense. 
Then Proposition~\ref{bigprop:usc_energy} implies $-M-\vep = \limsup_{k \to +\infty} \E_k(v_k) \leq \E_0(v)$; hence $v \in \CE^1(X_0,\om_0)$. 
Recall that from Proposition~\ref{prop:Hartogs_conv_subseq_family},  and Lemma~\ref{lem:L1_integral_converge},  
\begin{equation}\label{eq:pf_thmA_step1_sup}
    -C_{SL} \leq \sup_{X_0} v \leq 0.
\end{equation}
We have $-M-\vep \leq \E_0(\wv) \leq 0$ where $\wv = v - \sup_{X_0} v$.

\smallskip
As $\D_k(0)=0$ and $\D_k(u_k)\leq 0$, the convexity of $\D_k(\bullet)$ along the geodesic $s \mapsto g_k(s)$ (see Theorem \ref{thm:Berndtsson_cvx}) gives $\D_k(v_k)\leq 0$.
Hence,
\begin{align*}
    M+\vep = -\E_k(v_k) = \D_k(v_k) + \log\lt(\int_{X_k} e^{-v_k} c_k\dd \mu_k\rt)
    \leq \log\lt(\int_{X_k} e^{-v_k} c_k\dd \mu_k\rt).
\end{align*}
Taking limit on both sides, by Proposition~\ref{bigprop:DK}, \eqref{eq:pf_thmA_step1_sup}, and \eqref{eq:pf_thmA_coercivity}, we obtain 
\begin{align*}
    M+\vep &\leq \log \lt(\int_{X_0} e^{-v} c_0 \dd\mu_0\rt)
    = \log \lt(\int_{X_0} e^{- \wv} c_0 \dd\mu_0\rt) - \sup_{X_0} v \\
    &\leq \log \lt(\int_{X_0} e^{- \wv} c_0 \dd\mu_0\rt) + C_{SL}
    = - \D_0(\wv) - \E_0(\wv) + C_{SL}\\
    &\leq (1-A) (-\E_0(\wv)) + B + C_{SL}
    \leq (1-A) (M+\vep) + B + C_{SL}.
\end{align*}
Thus, we have $B+C_{SL}+A\vep =A(M+\vep) \leq B + C_{SL}$, which yields a contradiction.
\end{proof}

Fix $\vep>0$. By the claim, one can find $r_\vep > 0$ such that for all $t \in \BD_{r_\vep}$, 
\[
    \set{u \in \CE^1_{\nmlz}(X_t,\om_t)}{\D_t(u) \leq 0} 
    \subset \set{u \in \CE^1_{\nmlz}(X_t,\om_t)}{-\E_t(u) \leq M + \vep} 
    =: \CE^1_{M+\vep}(X_t,\om_t).
\]
By the compactness of $\CE^1_{M+\vep}(X_t,\om_t)$ with respect to the $L^1$-topology and the lower semi-continuity of the Ding functional, for each $t \in \BD_{r_\vep}$, there exists a minimizer in $\CE^1(X_t, \om_t)$ of $\D_t$. 
This function is a K\"ahler--Einstein potential on $X_t$ by \cite[Thm. 4.8]{BBEGZ_2019}. 
\end{proof}

\begin{rmk}
We give a remark on the uniqueness of K\"ahler--Einstein metrics. 
As $X_0$ admits a unique K\"ahler--Einstein metric, the connected component of the identity of the automorphism group $\Aut^\circ(X_0)$ is the identity (cf. \cite[Thm.~5.1]{BBEGZ_2019}). 
In particular, we have 
\[
    H^0(X_0, T_{X_0}) = \{0\}.
\] 
Thus, by the upper semi-continuity of $\BD \ni t \mapsto \dim H^0(X_t, T_{X_t})$ (see \cite[Ch.~3, Prop. 1.7]{Banica_Stanasila_1976}), we deduce that $\Aut^\circ(X_t) = \{\Id\}$ for any $t$ sufficiently close to $0$. 
Hence the uniqueness of K\"ahler--Einstein metrics again follows from \cite[Thm.~5.1]{BBEGZ_2019}.  
\end{rmk}

\subsection{A uniform $L^\infty$-estimate of K\"ahler--Einstein potentials}\label{sec:L_infty_est}
We prove here the uniform $L^\infty$-estimate part of Theorem~\ref{bigthm:L_infty_est}.

\smallskip
Fix $\om \in c_1(-K_{\CX/\BD})$ a smooth K\"ahler metric on $\CX$. 
Under Setting \ref{set:K_stab_fami}, Theorem \ref{bigthm:openness_KE} implies that, for all $t$ sufficiently close to $0$, $X_t$ admits a unique K\"ahler--Einstein potential $\vph_{\KE,t} \in \PSH(X_t, \om_t)$ such that  
\begin{equation}\tag{$\MA_t$}\label{eq:pf_thmA_MA_eqn}
    \frac{1}{V} (\om_t +\ddc_t \vph_{\KE,t})^n = e^{-\vph_{\KE,t}} c_t\mu_t
\end{equation}
where $\mu_t$ is the adapted measure on $X_t$ related to $\om_t$ (see Section \ref{sec:adapted_measure}), and $c_t = \frac{1}{\int_{X_t} \dd\mu_t}$. 
We denote by $\om_{\KE,t} = \om_t + \ddc_t \vph_{\KE,t}$ the unique K\"ahler--Einstein metric on $X_t$.

\begin{thm}\label{thm:L_infty_est}
Under the above setting, there exists $C_{\MA}>0$ such that for all $t$ sufficiently close to 0
$$
\osc_{X_t} \vph_{\KE,t}\leq C_{\MA}.
$$
\end{thm}

\begin{proof}
Up to shrinking $\BD$, one may assume that for all $t \in \BD$, each fibre $X_t$ admits a unique K\"ahler--Einstein potential $\vph_{\KE,t}$ solving \eqref{eq:pf_thmA_MA_eqn} (cf. Theorem \ref{thm:openness_KE}). Set $\wvph_{\KE,t} = \vph_{\KE,t} - \sup_{X_t} \vph_{\KE,t}$. By \cite[Thm.~4.8]{BBEGZ_2019}, $\wvph_{\KE,t}$ is the unique minimizer of $\D_t(\bullet)$ on $\CE^1_{\nmlz}(X_t, \om_t)$, and in particular $\D_t(\wvph_{\KE,t})\leq \D_t(0)=0$. Thus, thanks to Claim \ref{Claim}, there exists a uniform constant $C_\E>0$ such that
\begin{equation}\label{eq:pf_thmA_energy_lowerbdd}
    \E_t(\wvph_{\KE,t})\geq -C_\E
\end{equation}
for all $t\in \BD$.

\smallskip
Now, we shall prove a uniform $L^p$-estimate for $(\wvph_{\KE,t})_{t \in \BD}$ with respect to the measures $(\mu_t)_{t \in \BD}$.
By Proposition~\ref{prop:Hartogs_conv_subseq_family}, and Lemma~\ref{lem:L1_integral_converge}, up to extracting and relabeling, any sequence $(\wvph_{\KE,k})_k$ converges to a function $\vph \in \CE^1(X_0,\om_0)$ with $-C_{SL} \leq \sup_{X_0} \vph \leq 0$. For any $\gamma>1$, Proposition~\ref{bigprop:DK} gives
\[
    \int_{X_k}e^{-\gamma \wvph_{\KE,k}} \dd \mu_k 
    \longrightarrow 
    \int_{X_0}e^{-\gamma\vph} \dd \mu_0.
\]
In particular $\lVert e^{-\wvph_{\KE,k}}\rVert_{L^\gamma(\mu_k)}$ is uniformly bounded. Hence, an easy contradiction argument yields that for any $\gamma>1$ there exists a uniform constant $A_\gamma>0$ such that
\begin{equation}\label{eq:pf_thmA_MT_ineq}
    \lVert e^{-\wvph_{\KE,t}}\rVert_{L^\gamma(\mu_t)}\leq A_\gamma
\end{equation}
for all $t\in \BD$.

\smallskip
By \cite[Thm.~A]{DGG2020}, to establish a uniform $L^\infty$-estimate to the solution of \eqref{eq:pf_thmA_MA_eqn}, it suffices to verify two conditions: 
\begin{enumerate}[label=\textbf{H.\arabic*}]
    \item\label{item:DGG_hypo_Skoda} 
    There exist $\bt > 0$ and $B_\bt > 0$ such that for all $\psi_t \in \PSH(X_t,\om_t)$, 
    \[
        \int_{X_t} e^{-\bt (\psi_t - \sup_{X_t} \psi_t)} c_t \dd\mu_t \leq B_\bt;
    \]
    \item\label{item:DGG_hypo_Lp} 
    There exist $q > 1$ and $C > 0$ such that $\int_{X_t} e^{- q \vph_{\KE,t}} c_t\dd\mu_t \leq C$.
\end{enumerate}
Condition \ref{item:DGG_hypo_Skoda} follows from Theorem~\ref{thm:SL_and_Skoda_in_family} and \eqref{eq:klt_Lp_estimate}.
Indeed, choosing $\bt > 0$ such that $\frac{p}{p-1} \bt \leq \af$, by H\"older's inequality, we have
\[
    \int_{X_t} e^{- \bt(\psi_t - \sup_{X_t} \psi_t)} \dd\mu_t 
    \leq \lt(\int_{X_t} e^{- \frac{p}{p-1} \bt(\psi_t - \sup_{X_t} \psi_t)} \om_t^n\rt)^{(p-1)/p}
    \lt(\int_{X_t} f_t^p \om_t^n\rt)^{1/p} 
    \leq C_\af^{(p-1)/p} C_p^{1/p}.
\]
Condition \ref{item:DGG_hypo_Lp} is a direct consequence of \eqref{eq:pf_thmA_energy_lowerbdd} and \eqref{eq:pf_thmA_MT_ineq}.
This completes the proof.
\end{proof}

\subsection{Higher-order estimates away from the singular locus}\label{sec:higher_order_est}
The idea of getting a Laplacian estimate goes back to \cite{Paun_2008} (see also \cite[Appx.~B]{BBEGZ_2019}) on a fixed variety.
However, along a degenerate family $\pi: \CX \to \BD$ and its resolution $p: \CY \to \CX$, we would not have uniform control over the lower bound of bisectional curvature along each fibre even of smooth K\"ahler metrics, since the central fibre over the resolution could have several components. 
Therefore, instead of using Siu's inequality in \cite{Paun_2008, BBEGZ_2019}, it is more efficient to follow the strategy as in \cite[Prop.~2.1]{Guenancia_2016} with an argument by Chern--Lu inequality \cite{Chern_1968, Lu_1968}.
The argument is standard and it provides uniform Laplacian estimates away from $\CZ$ the singular locus of $\pi$:
for any compact subset $K \Subset \CX \setminus \CZ$, there is a constant $C_K > 0$ such that for all $t$ close to $0$,
\begin{equation}\label{eq:lap_est}
    \lt(\tr_{\om_t} \om_{\KE,t}\rt)_{|K_t} \leq C_K
\end{equation}
where $K_t = K \cap X_t$.
Then by \eqref{eq:lap_est} and standard bootstrapping argument (Evans--Krylov, and Schauder estimates), one can get uniform higher-order estimates
\begin{equation}\label{eq:higher_order_est}
    \norm{\vph_{\KE,t}}_{\CC^j(K_t)} \leq C_{K,j}
\end{equation}
for any $j \in \BN$.

\smallskip
With the above observations, we now prove the continuity of K\"ahler--Einstein currents in Theorem~\ref{bigthm:L_infty_est}:

\begin{cor}\label{cor:conti_KE_current}
Under the same setting as in Theorem~\ref{thm:L_infty_est}, $\vph_{\KE,t}\in \PSH(X_t,\om_t)$ $\CC^\infty$-converges in families to $\vph_{\KE,0}\in\PSH(X_0,\om_0)$, and for all $k \in \{0,1, \cdots, n\}$, the following map is continuous near $0 \in \BD$
\[
    t \longmapsto \om_{\KE,t}^k \w [X_{t}] \in (\CD_{(n-k,n-k)}(\CX))',
\]
where $(\CD_{(n-k,n-k)}(\CX))'$ is the space of currents of bidimension $(n-k,n-k)$ (bidegree $(k+1,k+1)$).
\end{cor}

\begin{proof}
Fix open subsets 
\[
    U_0 \Subset X_0^\reg
    \quad\text{and}\quad
    \CU \Subset \CX \setminus \CZ
    \quad\text{with}\,\,
    \CU \cap X_0 = U_0.
\]
Let $(B_\af, F_\af, \CB_\af)_\af$ be a locally trivial finite cover of $\CU$.
We set 
\[
    \vph_{\af,t} = F^\ast_{\af,t} \vph_{\KE,t}
    \quad\text{and}\quad
    \om_{\af,t} = F^\ast_{\af,t} \om_t.
\]
By \eqref{eq:higher_order_est}, $(\vph_{\af,t})_t$ is locally bounded for any $\CC^j$-norm on $B_\af$.
Then by the Arzela--Ascoli theorem, for every sequence $(t_k)_k$ which goes to zero, one can find a subsequence (that is still denoted by $(t_k)_k$ after relabeling) such that $(\vph_{\af,k})_k$ converges locally smoothly to a function $\vph_0 \in \PSH(B_\af, \om_0)$ on $B_\af$.
Arguing as in Section~\ref{subsec:conv_in_fami}, for every sequence $t_k \xrightarrow[k\to+\infty]{} 0$, we obtain a subsequence $(\vph_{\KE,k})_k$ which $\CC^\infty$-converges to a function $\vph_0 \in \PSH(X_0,\om_0)$ in the family sense. 
On $X_0^\reg$, $\vph_0$ satisfies 
\begin{equation}\label{eq:limit_eq}
    \frac{1}{V}(\om_0 + \ddc_0 \vph_0)^n = e^{-\vph_0} c_0 \mu_0.
\end{equation}
Both sides of \eqref{eq:limit_eq} trivially extend through $X_0^\sing$, by uniqueness of the solution to \eqref{eq:limit_eq}, one obtains $\vph_0 \equiv \vph_{\KE,0}$.
Therefore, $F^\ast_{\af,t} \vph_{\KE,t}$ (resp. $F^\ast_{\af,t} \om_{\KE,t}$) converges locally smoothly towards $\vph_{\KE,0}$ (resp. $\om_{\KE,0}$) on $B_\af$ when $t \to 0$.

\smallskip
Next, fix $\vep > 0$. 
Up to shrinking $\BD$, by Lemma~\ref{lem:small_cap_near_sing}, we find an open set $\CW_\vep$ which contains $\CZ$, and $\CAP_{\om_t}( \CW_\vep \cap X_t) < \vep$ for each $t \in \BD$. 
Then we take a open set $\CU_\vep \Subset \CX \setminus \CZ$ so that $\CU_\vep \cup \CW_\vep = \CX$. 
Set $U_0 = \CU_\vep \cap X_0$.
Let $(B_\af, F_\af, \CB_\af)_\af$ be a locally trivial finite open cover of $\CU_\vep$. 
Take cutoffs $(\chi_\af)_\af$ such that each $\chi_\af$ is supported in $\CB_\af$ and $\sum_\af {\chi_\af}_{|\CU_\vep} = 1$ up to shrinking $\BD$.
Write $\chi := \sum_\af \chi_\af$.

\smallskip
Fix $k \in \{0,\cdots,n\}$, and a test $(n-k,n-k)$-form $\eta$ on $\CX$.
We aim to prove 
\[
    \lim_{t \to 0} \int_\CX \eta \w \om_{\KE,t}^k \w [X_t] = \int_\CX \eta \w \om_{\KE,0}^k \w [X_0].
\]
With the cutoffs $(\chi_\af)_\af$, we can write
\[
    \int_\CX \eta \w \om_{\KE,t}^k \w [X_t] =
    \underbrace{\int_{X_t} (1-\chi) \eta \w \om_{\KE,t}^k}_{= \RN{1}_{t}} + \underbrace{\sum_\af \int_{\CB_\af} \chi_\af \eta \w \om_{\KE,t}^k}_{= \RN{2}_{t}}. 
\]
Since $F^\ast_{\af,t} \om_{\KE,t}$ (resp. $F^\ast_{\af,t}(\chi_\af\eta)_{|B_{\af,t}}$) converges locally smoothly to $\om_{\KE,0}$ (resp. $(\chi_\af\eta)_{|B_\af}$), we have 
\[
    \RN{2}_{t} = \sum_\af \int_{B_\af} F^\ast_{\af,t}(\chi_\af\eta)_{|B_{\af,t}} \w F^\ast_{\af,t} \om_{\KE,t}^k
    \xrightarrow[t \to 0]{} 
    \sum_\af \int_{B_\af} (\chi_\af\eta)_{|B_\af} \w \om_{\KE,0}^k
    = \RN{2}_{0}.
\]

\smallskip
On the other hand, 
\[
    |\RN{1}_t| \leq \int_{X_t \cap \CW_\vep} A_\eta \om_t^{n-k} \w \om_{\KE,t}^k 
\]
where $A_\eta > 0$ is a constant such that $|\eta|_\om \leq A_\eta \om^{n-k}$.
Let $\wvph_{\KE,t}$ be the sup-normalized K\"ahler--Einstein potential on $X_t$. 
By Theorem~\ref{bigthm:L_infty_est}, there is a uniform constant $C > 1$ such that 
\[
    \norm{\wvph_{\KE,t}}_{L^\infty(X_t)} \leq C.
\]
Note that 
\[
    \lt(\om_t + \ddc_t \frac{\wvph_{\KE,t}}{C}\rt)^n = \sum_{l=0}^n \binom{n}{l} \frac{(C-1)^l}{C^n} \om_t^{n-l} \w \om_{\KE,t}^l.
\]
Hence,
\[
    |\RN{1}_t| \leq \frac{A_\eta C^n}{(C - 1)^k} \CAP_{\om_t} (X_t \cap \CW_\vep)
    < \underbrace{\lt(\frac{A_\eta C^n}{(C - 1)^k}\rt)}_{=: A'_{\eta, k}} \vep.
\]
Since $A'_{\eta, k}$ is independent of $\vep$, 
\[
    \abs{\int_\CX \eta \w \om_{\KE,t}^k \w [X_t]
    - \int_\CX \eta \w \om_{\KE,0}^k \w [X_0]} \leq 3 A'_{\eta, k} \vep
\]
for all $t$ sufficiently close to $0$. This finishes the proof of Theorem~\ref{bigthm:L_infty_est}.
\end{proof}

\subsection{Moser--Trudinger inequalities.}\label{sec:Moser--Trudinger}
Variants of the Trudinger inequality on compact K\"ahler manifolds go back to the "Hypothèse fondamentale" of Aubin \cite{Aubin_1984} (proved in \cite{Berman_Berndtsson_2022}). In \cite[Prop. 4.11]{BBEGZ_2019} the coercivity of the Ding functional has been shown to be equivalent to the following Moser--Trudinger inequality: there exist $\delta>1,C>0$ such that
\[
    \lVert e^{-u}\rVert_{L^\delta(\mu)}\leq Ce^{-\E(u)}
\]
for any $u\in\CE^1(X,\om)$. We prove a uniform version of such Moser--Trudinger inequality.

\begin{prop}\label{prop:Moser--Trudinger}
There exist constants $\delta > 1$ and $C > 0$ such that for all $t$ close to $0$, for each $u \in \CE^1(X_t, \om_t)$,
\begin{equation}\label{eq:Uniform_Moser-Trudinger}
    \lVert e^{-u}\rVert_{L^\delta(\mu_t)}\leq Ce^{-\E_t(u)}.
\end{equation}
\end{prop}

\begin{proof}
By the klt version of \cite[Lem. 4.4]{DGG2020}, $\mu_t=f_t\om_t^n$ with $0\leq f_t\in L^p(X_t,\om_t)$ for a uniform $p>1$. Thus, combining \cite[Prop. 4.11]{BBEGZ_2019} and Theorem \ref{thm:SL_and_Skoda_in_family}, one deduces that (\ref{eq:Uniform_Moser-Trudinger}) is equivalent to finding $A>0,B>0$ such that for all $v \in \CE^1_{\KE,\nmlz}(X_t, \om_t)$, 
\begin{equation}\label{eq:Uniform_Coercivity}
    \D_t(u) \geq A(-\E_t(u)) - B.
\end{equation}
For any $t$ sufficiently close to $0$, we define
\begin{align*}
    \D^{\KE}_t(v) &= - \E_t^{\KE}(v) - \log\lt(\int_{X_t} e^{-v} e^{-\vph_{\KE}} \dd\mu_t\rt),\\
    \E^{\KE}_t(v) &= \frac{1}{(n+1) V} \sum_{j=0}^n \int_{X_t} v (\om_{\KE,t} + \ddc v)^j \w \om_{\KE,t}^{n-j},
\end{align*}
where 
\[
    v \in \CE_{\KE, \nmlz}^1(X_t,\om_t):= \set{u - \vph_{\KE,t}}{u \in \CE^1(X_t,\om_t) \text{ with } \sup_{X_t} (u - \vph_{\KE,t}) = 0}.
\]
Note that $\E_t(u) - \E_t^{\KE}(u-\vph_{\KE,t})=\E_t(\vph_{\KE,t})$ and $\D_t(u) - \D^{\KE}_t(u-\vph_{\KE,t}) = -\E_t(\vph_{\KE,t})$. Arguing as in Corollary \ref{cor:conti_KE_current} one can also show that $t\to \E_t(\vph_{\KE,t})$ is continuous near $0\in \BD$.

\smallskip
We claim that to prove (\ref{eq:Uniform_Coercivity}) it is sufficient to establish the following uniform coercivity for shifted functionals: there exist positive constants $A \in (0,1)$, $B > 0$ such that for all $t$ sufficiently close to $0$, for all $v \in \CE^1_{\KE,\nmlz}(X_t, \om_t)$, 
\begin{equation}\label{eq:Uniform_Coercivity2}
    \D^{\KE}_t(v) \geq A(-\E^{\KE}_t(v)) - B.
\end{equation}
Indeed, from (\ref{eq:Uniform_Coercivity2}) one can deduce that for all $u \in \CE^1(X_t, \om_t)$ with $\sup_{X_t} (u - \vph_{\KE,t}) = 0$, 
\begin{equation}\label{eq:coercivity_non_normalized}
    \D_t(u) \geq A(-\E_t(u)) - B + (A-1) \E_t(\vph_{\KE,t}).
\end{equation}
From Theorem~\ref{bigthm:L_infty_est}, there is a uniform constant $C_{\MA} > 0$ such that for all $t$ sufficiently close to $0$,  $\osc_{X_t} \vph_{\KE,t} < C_{\MA}$.
Hence, for all $v = u - \vph_{\KE,t} \in \CE^1_{\KE,\nmlz}(X_t, \om_t)$, we have 
\begin{equation}\label{eq:sup_KE_normalize}
    -C_{\MA} \leq \sup_{X_t} u \leq C_{\MA}.
\end{equation}
Set $\wu = u - \sup_{X_t} u \in \CE^1_{\nmlz}(X_t, \om_t)$ for all $u \in \CE^1(X_t, \om_t)$ with $\sup_{X_t} (u - \vph_{\KE,t}) = 0$. 
Combining (\ref{eq:coercivity_non_normalized}) and (\ref{eq:sup_KE_normalize}) we obtain
\begin{align*}
    \D_t(\wu) & \geq A(-\E_t(\wu)) - A \sup_{X_t} u - B + (A-1) \E_t(\vph_{\KE,t})\\
    & \geq A(-\E_t(\wu)) - A C_{\MA} - B + (A-1) C_{\MA},
\end{align*}
which concludes the claim.

\smallskip
It remains to prove (\ref{eq:Uniform_Coercivity2}). 
Suppose by contradiction that there is a sequence $v_k \in \CE^1_{\KE, \nmlz}(X_k, \om_k)$ such that 
\[
    \D^{\KE}_k(v_k) < \frac{1}{k+1} (-\E^{\KE}_k(v_k)) - 1.
\]

\smallskip
\noindent{\bf Case 1:} 
Assume that there is a constant $C > 0$ such that $\E^{\KE}_k(v_k) \geq - C$ for all $k$. 
Note that $v_k$ subconverges to $v_0 = u_0 - \vph_{\KE,0}$ in families where $u_0$ is the limit of $u_k = v_k + \vph_{\KE,k}$.
Propositions~\ref{bigprop:usc_energy}, \ref{bigprop:DK} give $\E^{\KE}_0(v_0)\geq \limsup_{k \to +\infty} \E^{\KE}_k(v_k)\geq -C$ and
\[
    \D^{\KE}_0(v_0) \leq \liminf_{k \to +\infty} \D^{\KE}_k(v_k) \leq -1.
\]
However, $\D^{\KE}_0(0) = 0$ is the minimum, which yields a contradiction. 

\smallskip
\noindent{\bf Case 2:} Up to extracting a subsequence we can now assume that $\E^{\KE}_k(v_k) \to - \infty$ as $k \to +\infty$.
Set 
\[
    d_k := d_1(\vph_{\KE,k}, u_k) = \E_k(\vph_{\KE,k}) - \E_k(u_k)
\] 
and let $(u_k(s))_{s \in [0,d_k]}$ be the unit-speed geodesic connecting $\vph_{\KE,k}$ and $u_k$.
Since $u_k \leq \vph_{\KE,k}$, we have $u_k \leq u_k(s) \leq \vph_{\KE,k}$ for all $s \in [0, d_k]$ and thus, $u_k(s)$ still belongs to $\CE^1_{\KE,\nmlz}(X_k, \om_k)$.
We put $v_k(s) = u_k(s) - \vph_{\KE,k}$.

\smallskip
Fix an arbitrary $M > 0$. 
By Theorem \ref{thm:Berndtsson_cvx}, we have 
{\small
\begin{equation*}
    0 \leq \frac{\D^{\KE}_k(v_k(M)) - \D^{\KE}_k(v_k(0))}{M}
    \leq \frac{\D^{\KE}_k(v_k(d_k)) - \D^{\KE}_k(v_k(0))}{d_k}< \frac{1}{d_k} \lt(\frac{-\E^{\KE}_k(v_k)}{k+1} - 1\rt) = \frac{1}{k+1} - \frac{1}{d_k}.
\end{equation*}
}%
Therefore, 
\[
    \D^{\KE}_k(v_k(M)) \xrightarrow[k \to +\infty]{} 0 = \D^{\KE}_0(0).
\]
Let $v_0^M$ (resp. $u_0^M$) be a limit of a subsequence of $(v_k(M))_k$ (resp. $(u_k(M))_k$). Again Propositions \ref{bigprop:usc_energy}, \ref{bigprop:DK} lead to
\[
    \D^{\KE}_0(v_0^M) \leq \liminf_{k \to +\infty} \D_k^{\KE}(v_k(M)) = 0.
\]
This implies that $v_0^M=c_0^M\in \BR$ since minimizers of $\D_0^{\KE}$ are constants.
By Propositions~\ref{prop:Hartogs_conv_subseq_family} and \ref{lem:L1_integral_converge}, we have $- C_{\MA} - C_{SL} \leq \sup_{X_0} u_0^M \leq C_{\MA}$. 
One then obtains 
\[
    c_0^M=\sup_{X_0} v_0^M \in [- 2C_{\MA} - C_{SL}, 2C_{\MA}].
\]

\smallskip
On the other hand, by Proposition~\ref{bigprop:DK}, we have 
\[
    \int_{X_k} e^{-u_k(M)} \dd\mu_k 
    \xrightarrow[k \to +\infty]{} 
    \int_{X_0} e^{-u_0^M} \dd\mu_0 = \int_{X_0} e^{-\vph_{\KE,0} - c_0^M} \dd\mu_0 = e^{-c_0^M}.
\]
Therefore, 
\begin{align*}
    \E_k(u_k(M)) 
    &= - \D_k(u_k(M)) - \log \int_{X_k} e^{-u_k(M)} \dd\mu_k \\
    &= - \D^{\KE}_k(v_k(M)) + \E_k(\vph_{\KE,k}) - \log \int_{X_k} e^{-u_k(M)} \dd\mu_k
    \xrightarrow[k \to +\infty]{} \E_0(\vph_{\KE,0}) + c_0^M.
\end{align*}
and this implies that
\[
    M = d_1(\vph_{\KE,k}, u_k(M)) 
    = \E_k(\vph_{\KE,k}) - \E_k(u_k(M)) 
    \xrightarrow[k \to +\infty]{} - c_0^M \in [-2C_{\MA}, 2C_{\MA} + C_{SL}].
\]
Choosing $M$ large enough yields a contradiction and concludes the proof.
\end{proof}

\subsection{Generalization to log Fano pairs}\label{sec:log pairs}
Let $\pi: \CX \to \BD$ be a family satisfying Setting~\ref{sett:general_sett_normal_fibre}.
Let $\CD$ be an effective $\BQ$-divisor.
Suppose that
\begin{enumerate}
    \item $K_{\CX/\BD} + \CD$ is a $\BQ$-line bundle;
    \item every irreducible component of $\CD$ surjects onto $\BD$;
    \item $(X_0,D_0)$ is klt where $D_t := \CD_{|X_t}$.
\end{enumerate}

\begin{rmk}
Similar to the case $\CD = 0$, $(\CX, X_0 + \CD)$ is plt near $X_0$ by the inversion of adjunction \cite[Thm.~5.50]{Kollar_Mori_1998}.
Namely, for any log-resolution $p: \CY \to \CX$, we have
\[
    K_\CY + p^{-1}_\ast(X_0 + \CD) 
    = p^\ast(K_\CX + X_0 + \CD) + \sum_i a_i E_i 
\]
with $a_i > -1$ for all $i$. 
In particular, for all $t$ close to $0$, $(X_t,D_t)$ is still klt.
\end{rmk}

Fix $\om \in c_1(-K_{\CX/\BD} - \CD)$ a K\"ahler metric. For each $t \in \BD$, the log K\"ahler--Einstein metrics $\om_{\KE,t} = \om_t + \ddc_t \vph_{\KE,t}$ on $X_t$ are defined through the following Einstein equation:
\[
    \Ric(\om_{\KE,t}) = \om_{\KE,t} + [D_t].
\]
Equivalently, up to translation, the log K\"ahler--Einstein potentials $\vph_{\KE,t}$ solve the complex Monge--Amp\`ere equations
\[
    \frac{1}{V} (\om_t + \ddc_t \vph_{\KE,t})^n = e^{-\vph_{\KE,t}} \mu_t
\]
where the adapted measure is associated with $\om_t$ and $D_t$ as follows.
We assume that $m$ is an integer such that $m (K_{\CX/\BD} + \CD)$ is Cartier.
The adapted measure $\mu_t$ is defined as 
\[
    c_t\ii^{n^2} \lt(\frac{(\Om_t \w \overline{\Om_t})}{|\Om_t|_{h_t}^2}\rt)^{\frac{1}{m}} e^{- {\psi_\CD}_{|X_t}}
\]
where 
\begin{itemize}
    \item $h$ is a metric of $m (K_{\CX/\BD} + \CD)$,
    \item $\Om$ is a local generator of $m(K_{\CX/\BD} + \CD)$ on an open set $U \subset \CX$ and $\Om_t$ is the restriction of $\Om$ on $U \cap X_t$, 
    \item $\psi_\CD$ is a quasi-psh function on $\CX$ such that $\ddc \psi_\CD = [\CD]$,
    \item $c_t$ is chosen so that $\mu_t$ is a probability measure on $X_t$.
\end{itemize}
By the klt version of \cite[Lem. 4.4]{DGG2020}, $\mu_t=f_t\om_t^n$ for $f_t\in L^p(X_t)$ for $p>1$ with uniformly bounded $L^p$-norm. Moreover, $c_t$ is continuous in $t$ and bounded away from $0$ and $+\infty$.

\noindent{L}et $\Aut(X_t,D_t)$ be the stabilizer of $D_t$ in $\Aut(X_t)$.  Uniqueness of log K\"ahler--Einstein metrics holds modulo $\Aut^\circ(X_t,D_t)$, the connected component of the identity in $\Aut(X_t,D_t)$ (cf. \cite[Thm.~5.1]{BBEGZ_2019}).

\smallskip
Arguing as in Sections \ref{sec:L_infty_est} and \ref{sec:higher_order_est}, one obtains the following results.

\begin{thm}
Under the above setting, assume $\Aut^\circ(X_0,D_0)$ is discrete. 
If $(X_0, D_0)$ admits a log K\"ahler--Einstein metric, then for all $t$ sufficiently close to $0$, $(X_t, D_t)$ admits a log K\"ahler--Einstein metric.
\end{thm}

\begin{rmk}
The convexity along geodesics of the log Ding functional 
\[
    \D_t:\CE^1(X_t,\om_t)\to \BR,
\] 
constructed with respect to the aforementioned adapted probability measure $\mu_t$, is still given by \cite[Lem. 4.6]{BBEGZ_2019}. 
Similarly, the equivalence between \ref{item:Summ_Thm_1} and \ref{item:Summ_Thm_2} in Theorem \ref{Summ_Thm} in the log setting can be found in \cite[Thm.~5.5]{Dinezza_Guedj_2018} (see also \cite[Thm.~2.2]{Darvas_2017}).
\end{rmk}

\begin{thm}\label{thm:L_infty_log}
Under the above setting, there exists a constant $C_{\MA}>0$ such that for all $t$ sufficiently close to $0$,
\[
    \osc_{X_t}\vph_{\KE,t}\leq C_{\MA}.
\]
\end{thm}
Similarly, the results stated in Sections \ref{sec:higher_order_est}, \ref{sec:Moser--Trudinger} extend to the log setting of Theorem \ref{thm:L_infty_log}. 
Namely one obtains
\begin{itemize}
    \item uniform higher-order estimates for the log K\"ahler--Einstein potentials on compact sets $K \Subset \CX^\reg \setminus \CD$;
    \item the analog of Corollary \ref{cor:conti_KE_current} replacing in Definition \ref{defn:conv_in_family} $\CZ$ with $\CZ\cup \CD$;
    \item the uniform version of the Moser--Trudinger inequalities of Proposition \ref{prop:Moser--Trudinger} for the adapted measures $\mu_t$ defined at the beginning of this section.
\end{itemize}

\bibliographystyle{smfalpha_new}
\bibliography{biblio}

\end{document}